\documentclass[11pt, english]{smfart}

\usepackage[T1]{fontenc}
\usepackage[english, french]{babel}

\usepackage{amssymb}
\usepackage{amsmath}
\usepackage{smfthm}
\usepackage{mathrsfs}

\usepackage{graphicx}
\usepackage{array}

\title[Endless continuability and convolution product]{Endless
  continuability and convolution product}  

\alttitle{Prolongement sans fin et produit de convolution}
 
\author{Yafei Ou}
\address{SJTU-ParisTech Elite Institute of Technology,
800 Dongchuan Rd, Shanghai 200240, P.R.China.}
\email{oyafei@sjtu.edu.cn}
\urladdr{}
\author{Eric Delabaere}
\address{Universit\'e d'Angers, laboratoire LAREMA UMR CNRS 6093,
Universit\'e d'Angers, 2 Boulevard Lavoisier, 49045 Angers Cedex 01,
France.}
\email{eric.delabaere@univ-angers.fr}
\urladdr{}

\theoremstyle{definition}

\newcommand{\zetadot}{ {\stackrel{\bullet}{\zeta}} }
\newcommand{\xidot}{ {\stackrel{\bullet}{\xi}} }
\newcommand{\omegadot}{ {\stackrel{\bullet}{\omega}} }

\newcommand{\thetadot}{ {\dot{\theta}} }

\newcommand{\sect}{ {\boldsymbol{\mathfrak{s}}} }

\begin{document}
\frontmatter
 
\begin{abstract}
We provide a rigorous analysis for the so-called endlessly continuable
germs of holomorphic functions or in other words, the Ecalle's
resurgent functions. We follow and complete an approach due to 
 Pham,  based on the notion of discrete
filtered set $\Omega_\star$  and the associated Riemann surface defined
as the space of $\Omega_\star$-homotopy classes of paths. Our main
contribution consists in a complete though simple proof of the stability under
convolution product of the space of endlessly continuable germs.
\end{abstract}

%\begin{altabstract}
%\end{altabstract}



%\keywords{Sommation de Borel, s\'eries de factorielles.}
%\altkeywords{Borel-resummation, factorial series.}
%\subjclass{30E15, 40Gxx} 
 
\maketitle
\tableofcontents 

\mainmatter
 
\section{Introduction}\label{chapt5}

Today, the resurgence theory of Ecalle has demonstrated its
efficiency in many instances for dealing with divergent series arising from differential
and difference equations, PDEs, semiclassical analysis,
etc.. see for instance \cite{Ec84, DP99, Kawai-96, Kawai-004, Cost009,
  Del014, S014} and references
therein. In particular, this theory provides the necessary tools for
understanding the nonlinear Stokes phenomena in asymptotic analysis
and leads up to both theoretical results and even numerical methods,
see e.g. \cite{DH002}.

In its simplest form, the Ecalle's theory  generalizes the Borel
resummation theory, whose main objects are 1-Gevrey formal series
$\displaystyle \widetilde{\varphi}(z) = \sum_{n \geq 1} \frac{a_n}{z^n} \in
\mathbb{C}[[z^{-1}]]$, that is the formal Borel transform $\displaystyle  
\widehat{\varphi}(\zeta)=\sum_{n \geq 1} \frac{a_n}{(n-1)!} \zeta^n \in
\mathbb{C}\{\zeta\}$ defines a germ of holomorphic functions at
the origin. The formal series $\widetilde{\varphi}$ is said to be
resurgent if $\widehat{\varphi}$  is ``endlessly continuable''. 
Roughly speaking, $\widehat{\varphi}$ 
is endlessly continuable if $\widehat{\varphi}$  can be  analytically continued
along any path on $\mathbb{C}$  avoiding a set of possibly singular
points, each of them being locally isolated.
However, this set may be everywhere dense and is so for many important
applications, in particular those stemming from high energy physics,
string theory and related models \cite{DDP93, DDP97, Aniceto012,
  Marino009, Dunne014}

It turns out that the Cauchy product
$(\widetilde{\varphi}.\widetilde{\psi})(z)$ 
of 1-Gevrey formal series becomes the convolution
product $\displaystyle \widehat{\varphi} \ast \widehat{\psi}(\zeta)= \int_0^\zeta 
\widehat{\varphi}(\eta) \widehat{\psi}(\zeta-\eta) \, d\eta$
of germs of  holomorphic functions by Borel transformation. Since the Ecalle's theory
aims at analyzing nonlinear problems, it is an essential demand that the convolution
product preserves the notion of endless
continuability. This ends up with the definition of various subalgebras of resurgent
functions in the so-called Borel $\zeta$-plane and their counterpart in the
initial $z$-plane, as well as alien operators designed for encoding
the singularities by microlocalisation and describing the nonlinear
Stokes phenomena \cite{Ec85, Ec93-1, CNP1, S014, Del014}.

As a matter of fact, there exist  various  definitions  of  ``endless
continuability'' in the literature. 
The more general one is that of Ecalle \cite{Ec85,
  Ec93-1}. Another one is due to Pham {\em et al.} \cite{CNP2,
  CNP1}, is  easier to handle with and is based on 
the construction of a Riemann surface governed by the datum of a discrete filtered
set. This provides the notion of endless Riemann surface. A germ 
of holomorphic functions that can be analytically continued to such a
Riemann surface is by definition endlessly continuable. This is this
second approach that we use in this paper.

As already said, the endless continuability  and its stability under  convolution
product are  key-properties at the very root of the
resurgence theory. Unfortunately, the existing proofs for the
stability in its full generality \cite{Ec85, CNP1} are difficult and
subject to controversy.
Our main goal in this article is to show in a simple and rigorous way 
that the convolution product
of any two endlessly continuable functions is endlessly
continuable. Though inspired by \cite{CNP1}, our methods differ from
these authors  for key-arguments. The method that we present in this paper can be seen
as an extension of  ideas
detailed in \cite{OU010, S012, S014} for the case of holomorphic functions 
that can be analytically continued along any path avoiding closed discrete
subsets of $\mathbb{C}$. Consequently, we thought that our results were
worth the attention of the specialists in this field of research.

The paper is organized as follows.
We introduce the notion of discrete filtered
sets and their associated Riemann surfaces that we study properly (Sect. \ref{filsetR}). 
We define endless Riemann surfaces and
endlessly continuable germs of holomorphic functions
 and we make a link with 
the endlessly continuable functions of Ecalle (Sect. \ref{CNP-s-end}).
The main result of the paper concerns the stability under
convolution product and this is detailed in
Sect. \ref{mainsectionend}. We end the paper with some open problems.

\section{Discrete filtered set and associated Riemann
  surface}\label{filsetR} 
 
\subsection{Discrete filtered sets}

The following definitions are adapted   from \cite{CNP1}.

\begin{defi}\label{FiltDis}
A \textbf{discrete filtered set} $\Omega_\star$ centred at $\omega \in
\mathbb{C}$ is an 
increasing sequence of finite sets $\Omega_L \subset \mathbb{C}$, $ L
>0$,  such that :
\begin{itemize}
\item for any $L >0$, $\Omega_L$ belongs to the open disc centred at
  $\omega$ with radius $L$; 
\item if $L_1 \leq L_2$ then $\Omega_{L_1} \subseteqq \Omega_{L_2}$;
\item for $L>0$ small enough, $\Omega_L = \{\omega\}$.
\end{itemize}
For $L>0$, we denote $\Omega_L^\star = \Omega_L  \setminus
\{\omega\}$. The number  $\rho_{\Omega_\star}(\omega) = \sup \{L>0 \mid
\Omega_L^\star=\emptyset\}$ is called the \textbf{distance of
  $\omega$ to $\Omega_\star$}.
\end{defi}

\begin{defi}\label{UnionFiltDis}
Let  $\Omega_{\star}$ and $\Omega_{\star}^\prime$ be  two  discrete 
filtered sets centred at  $\omega \in \mathbb{C}$. Their 
 \textbf{union} $\Omega_\star \cup \Omega_\star^\prime$ is the discrete
filtered set centred at  $\omega$ defined by : for every $L>0$, $(\Omega_\star \cup
\Omega_\star^\prime)_L = \Omega_L \cup \Omega_L^\prime$. Their \textbf{sum}
$\Omega_\star + \Omega_\star^\prime$ is the filtered set centred at
$\omega$  defined by : for every $L>0$,
$(\Omega_\star + \Omega_\star^\prime)_L =  \{-\omega+\Omega_L + \Omega^\prime_L \}  \cap
D(\omega,L)$. Their \textbf{fine sum}
$\Omega_\star \ast \Omega_\star^\prime$ is the filtered set centred at
$\omega$ given by : for every $L>0$,
$(\Omega_\star \ast \Omega_\star^\prime)_L =  \{\zeta=
-\omega+ \omega_1+\omega_2 \mid \omega_1 \in \Omega_{L_1}, \omega_2
\in \Omega^\prime_{L_2}, L_1+L_2 =L \}$. 
\end{defi}

If $\Omega_\star$ is a discrete filtered set, we remark that
$\bigcup_{L>0} \Omega_L$ can be dense in $\mathbb{C}$ as it is shown
in the following example.

\begin{exem}
Assume that $\omega_1 \in \mathbb{C}^\star$ and define
\begin{itemize}
\item for any $L \in ]0, |\omega_1|]$, $\Omega_L = \{0\}$,
\item for any $n \in \mathbb{N}^\star$ and any $L \in
  ]n|\omega_1|, (n+1)|\omega_1|]$, $\Omega_L = \{0, \pm \omega_1, \cdots,
  \pm n \omega_1\}$.
\end{itemize}
This define a discrete filtered set $\Omega_{1\star}$ centred at $0$.

Assume now that $\omega_1, \omega_2, \omega_3 \in \mathbb{C}^\star$
are rationally independent, that is linearly independent over
$\mathbb{Z}$. We consider  the
three discrete filtered sets $\Omega_{1\star}$,  $\Omega_{2\star}$
and $\Omega_{3\star}$ centred at $0$ defined as above. We note
$\Omega_\star = \Omega_{1\star} + \Omega_{2\star}+ \Omega_{3\star}$ their
sum. Then $\bigcup_{L>0} \Omega_L$ is everywhere dense   in
$\mathbb{C}$. The conclusion is the same when $\Omega_\star =
\Omega_{1\star} \ast \Omega_{2\star} \ast \Omega_{3\star}$ is
defined by fine sums.
\end{exem}

For a given discrete filtered set $\Omega_\star$ centred at
$\omega$, its iterated fine sums ${\sum_{n \ast} \Omega_\star =
\renewcommand{\arraystretch}{0.5}
\begin{array}[t]{c}
\displaystyle \underbrace{\Omega_\star \ast \cdots \ast \Omega_\star}\\
{\scriptstyle n \mbox{ times}}
\end{array}
\renewcommand{\arraystretch}{1}}$ makes a direct system (for the
injections \linebreak ${\left(\sum_{n \ast} \Omega_\star\right)_L \hookrightarrow
\left(\sum_{(n+1) \ast} \Omega_\star\right)_L}$, for every $L>0$). The
fine sums enjoyes the following property (the proof is left to the reader):

\begin{prop}
Let $\Omega_\star$ be a discrete filtered set $\Omega_\star$ centred at
$\omega$. Then the direct limit $\Omega_\star^\infty = 
\displaystyle\lim_{\rightarrow}\sum_{n \ast}
\Omega_\star$ is a discrete filtered set at
$\omega$.
\end{prop}

\begin{defi}
The discrete filtered set $\Omega_\star^\infty$ is called the
\textbf{saturated} of $\Omega_\star$.
\end{defi}

\subsection{Reminder about paths}

In what follows, a path $\lambda$ in a topological space $X$ is any
continuous function $\lambda: [a,a+l] \to X$, where $[a,a+l] \subset
\mathbb{R}$ is a (compact) interval possibly 
reduced to~$\{a\}$. We often work with standard
paths\index{Path!standard path}, that is paths
defined on $[0,1]$. The path  ${ \underline{\lambda} : t \in [0,1] \mapsto
\lambda(a+tl)}$ is the standardized path\index{Path!standardized path}
of  $\lambda$. For  two paths 
$\lambda_1: [a,a+l] 
\to X$, $\lambda_2: [b,b+k] \to X$ so that
$\lambda_1(a+l)=\lambda_2(b)$, one defines their product (or
concatenation) by 
$${ \displaystyle \lambda_1 \lambda_2 : t \in [a,a+l+k]\mapsto 
\left\{
\begin{array}{l}
\lambda_1(t), \, t \in [a,a+l]\\
\lambda_2(t-a-l+b), \, t \in [a+l,a+l+k]
\end{array}
\right.} 
$$
 When the two paths $\lambda_1$, $\lambda_2$ have same extremities,
they are homotopic when there exists a continuous map $H : [0,1]\times
[0,1] \to X$ that realizes a homotopy between the standardized
paths\index{Path!product}
$\underline{\lambda}_1$ and $\underline{\lambda}_2$.
We recall that any path  $\lambda:I \to \mathbb{C}$
can be uniformaly approached by $\mathcal{C}^\infty$-paths.
When $\lambda : I \to \mathbb{C}$ is piecewise $\mathcal{C}^1$, we denote its 
length\index{Path!length} by 
${ \displaystyle \mathcal{L}_\lambda  }$.

\subsection{$\Omega_\star$-allowed path, $\Omega_\star$-homotopy}

The following definitions are inspired from  \cite{CNP1} but for
slight modifications\footnote{These definitions are less general than
  those of  \cite{CNP1} but sufficient in practice as far as we know.}.

\begin{defi}\label{Allpath}
For $\Omega_\star$  a discrete filtered set centred at
$\omega \in \mathbb{C}$,
one denotes by $\mathfrak{R}_{\Omega_\star}(L)$ the set of paths  $\lambda : I \rightarrow
\mathbb{C}$ starting from $\omega$ and such that :
\begin{itemize}
\item $\lambda$ is  $\mathcal{C}^1$ piecewise and its  length satisfies 
$\mathcal{L}_\lambda <L$;
\item $\lambda$ is the constant path or there exists $t_0 \in [0,1[$
  such that  $\underline{\lambda}([0,t_0])=\{ \omega
  \}$ and $\underline{\lambda}(]t_0,1]) \subset D(\omega,L) \setminus \Omega_{L}$.
\end{itemize}
A path $\lambda$ is said to be $\Omega_\star$-allowed if $\lambda \in
\mathfrak{R}_{\Omega_\star}(L)$ for some $L>0$.
We denote by ${ \mathfrak{R}_{\Omega_\star} = \bigcup_{L>0}
\mathfrak{R}_{\Omega_\star}(L) }$ the set of $\Omega_\star$-allowed paths.
\end{defi}

\begin{defi}\label{Allhomotopy}
Let  $\Omega_\star$ be a discrete filtered set centred at
$\omega \in \mathbb{C}$. A continuous map
$\displaystyle H : (s,t) \in  [0,1]^2 \mapsto H_t(s) \in \mathbb{C}$
is a $\Omega_\star$-homotopy if $H$ has 
a continuous partial derivative $\displaystyle \frac{\partial
H}{ \partial s}$ and, for every $t \in [0,1]$, the
path $H_t$ is $\Omega_\star$-allowed.

Two $\Omega_\star$-allowed paths $\lambda_0$ and $\lambda_1$ with same
extremities are $\Omega_\star$-homotopic
when there exists a $\Omega_\star$-homotopy that realises a homotopy
between the standardized paths 
$\underline{\lambda}_0$ and  $\underline{\lambda}_1$.\\
For $\lambda$ a $\Omega_\star$-allowed path, we denote by $\mathrm{cl}(\lambda)$ its
equivalence class for the relation $\sim_{\Omega_\star}$ of  
$\Omega_\star$-homotopy  of paths in  $\mathfrak{R}_{\Omega_\star}$ with fixed
extremities.
\end{defi}

It is quite important to understand what is the
$\Omega_\star$-homotopy and we make the following remark that we
formulate as a lemma:

\begin{lemm}\label{RemarksurHomotopy}
Assume that $\Omega_\star$ be a discrete filtered set centred at
$\omega \in \mathbb{C}$ and let $H$ be a $\Omega_\star$-homotopy.
Then there exists a good\footnote{By ``good'', we mean that the
covering has finite  elements, that each of these element $I_i$ is a
connected interval and that there are no 3-by-3 intersections.} 
 open covering $(I_i)_{0 \leq i \leq n}$ of $[0,1]$ 
and  real positive numbers $L_0, L_1, \cdots , L_n$
such that, 
\begin{itemize}
\item for every $i = 0, \cdots, n$ and every $t \in I_i$, 
$H_t$ belongs to $\mathfrak{R}_{\Omega_\star}(L_i)$.
\item for every $i = 0, \cdots, n-1$ and for every $t \in I_i \cap
  I_{i+1}$, $H_t \in \mathfrak{R}_{\Omega_\star}(L_i) \cap
  \mathfrak{R}_{\Omega_\star}(L_{i+1})$. 
\end{itemize}
\end{lemm}

\begin{proof}
From the very definition of a discrete filtered set,
 one can define an increasing
sequence of real numbers $0=l_{-1} < l_0 <  l_1 < l_2 < \cdots $ with the
properties:
\begin{itemize}
\item $\Omega_{l_0} = \{\omega\}$ and for any integer $i \geq 1$, 
$\Omega_{l_{i-1}} \subsetneqq \Omega_{l_{i}}$;
\item  $ \Omega_{L}  =
  \Omega_{l_{i}}$ for every $L \in]l_{i-1}, l_i]$, $i \in \mathbb{N}$.
\end{itemize}
Pick any $t_\star \in [0,1]$ and assume that
$\mathcal{L}_{H_{t_\star}} \in [l_{i-1},l_i[$ for some $i \in
\mathbb{N}$. 
One has $H_{t_\star} \in
\mathfrak{R}_{\Omega_\star}(l_i)$ necessarily since $H_{t_\star}$
is $\Omega_\star$-allowed. We say that $H_{t} \in
\mathfrak{R}_{\Omega_\star}(l_i)$ for any $t \in [0,1]$ close enough to
$t_\star$. Indeed, we remark that the map $t \in [0,1]
\mapsto \displaystyle \mathcal{L}_{H_t}$ is continuous
 because of the existence and the 
continuity of the partial derivative $\displaystyle \frac{\partial
H}{ \partial s}$. Thus, if $\mathcal{L}_{H_{t_\star}} \in
]l_{i-1},l_i[$, then $\mathcal{L}_{H_{t}} \in
]l_{i-1},l_i[$ for $t$ close enough to $t_\star$. Now if
$\mathcal{L}_{H_{t_\star}} =  l_{i-1}$ ($i \geq 1$), then $\mathcal{L}_{H_{t}} \in
]l_{i-2},l_i[$ for $t$ close enough to $t_\star$. However,
$\mathcal{L}_{H_{t}}$ belongs also to
$\mathfrak{R}_{\Omega_\star}(l_i)$ for $t$ close enough to $t_\star$
because of the continuity of
$H$ (the euclidean distance $d(H_{t_\star}, \Omega_{l_i})$
of the path $H_{t}$ to the set $\Omega_{l_i}$ is $>0$, so does 
$d(H_{t}, \Omega_{l_i})$ for $t$ close enough to $t_\star$).
 This way one gets an open covering of $[0,1]$ from which one deduces a finite
open covering by compacity. One  easily concludes.
\end{proof}

\begin{figure}[thp]
  \centering\includegraphics[scale=.63]{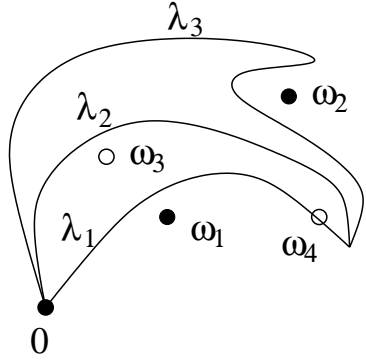}\\
  \centering\caption{We assume that $\Omega_\star$ is a  discrete
    filtered set centred at $0$. For $0 < L_1 < L_2$,  $\Omega_{L_1} =
    \{0,\omega_1, \omega_2\}$, $\Omega_{L_2} = \Omega_{L_1} \cup
    \{\omega_3, \omega_4\}$. The paths ${ \lambda_1, \lambda_2 \in
    \mathfrak{R}_{\Omega_\star}(L_1) }$ are  $\Omega_\star$-homotopic,
    the paths $\lambda_2, \lambda_3 \in 
    \mathfrak{R}_{\Omega_\star}(L_2)$ are $\Omega_\star$-homotopic,
    thus $\lambda_1$ and $\lambda_3$ are $\Omega_\star$-homotopic
    despite the fact that $\lambda_1 \notin \mathfrak{R}_{\Omega_\star}(L_2)$.
  }\label{Yafei-E-4}
\end{figure}

From lemma \ref{RemarksurHomotopy}, observe that  if $L_2 > L_1$, a path $\lambda_1 \in
\mathfrak{R}_{\Omega_\star}(L_1)$ can be  $\Omega_\star$-homotopic to another
path  $\lambda_2 \in  \mathfrak{R}_{\Omega_\star}(L_2)$ and at the
same time $\lambda_1$ not being homotopic to $\lambda_2$ in the usual way,
when both are seen as paths in  $\mathfrak{R}_{\Omega_\star}(L_2)$. 
Even, we may have $\lambda_1 \notin
\mathfrak{R}_{\Omega_\star}(L_2)$, see Fig. \ref{Yafei-E-4}.

\subsection{Riemann surface associated with a discrete filtered set}

\begin{defi}\label{homotopyclass}
Let $\Omega_\star$ be a discrete filtered set centred at
$\omegadot \in \mathbb{C}$. We set:
$$\mathscr{R}_{\Omega_\star} = \{\zeta = \mathrm{cl}(\lambda) \, \mid \, \lambda \in
\mathfrak{R}_{\Omega_\star}\} 
\hspace{2mm} \mbox{and} \hspace{2mm} \mathfrak{p} : \zeta = 
\mathrm{cl}(\lambda) \mapsto
\zetadot = \underline{\lambda}(1) \in \mathbb{C}.
$$
\end{defi}

Remark that  $\mathfrak{p}^{-1}(\omegadot)$ is reduced to a single
point $\omega = \mathrm{cl}(\mbox{constant path})$. 
 This is why  one usually considers $\mathscr{R}_{\Omega_\star}$ as a pointed
space $(\mathscr{R},\omega)$.

Let $\Omega_\star$  be a discrete filtered set centred at
$\omegadot \in \mathbb{C}$, and set $\omega =\mathfrak{p}^{-1}(\omegadot)$.
One can  endow $\mathscr{R}_{\Omega_\star}$ with a separated topology, a basis $\mathscr{B}=
\{ \mathscr{U} \}$ of open 
sets\index{Riemann surface $\mathscr{R}_{\Omega_\star}$!topology $\mathscr{B}$} 
defining this topology  being given as follows\footnote{This topoloy
  is not detailed in \cite{CNP1}.}. (We adapt the
classical construction of a universal covering
\cite{For}). Let us consider a
point $\zeta \in \mathscr{R}_{\Omega_\star}$.
\begin{itemize}
\item Assume that  
$\zeta = \omega$. For some $L>0$ we consider $\stackrel{\bullet}{\mathscr{U}} \subset
  D(\omegadot,L) \setminus \Omega_{L}^\star$  a star-shaped domain with
  respect to $\omegadot$. Let
 $\mathscr{U} \subset \mathscr{R}_{\Omega_\star}$ be the set of all
 $\xi = \mathrm{cl}(\lambda)$ where  $\lambda \in
  \mathfrak{R}_{\Omega_\star}(L)$ is any path  ending at $\xidot \in
  \stackrel{\bullet}{\mathscr{U}}$ and whose 
  image is the line segment $[\omega,\xidot]$. (For a given $\xidot$, 
the length of these paths is
  $|\xidot| < L$ and all these paths belong to the
  same $\Omega_\star$-homotopy class).
\item Suppose that $\zeta \neq \omega$. We choose  a path $\lambda_1 \in
  \mathfrak{R}_{\Omega_\star}(L)$ such that $\mathrm{cl}(\lambda_1) =
  \zeta$.  For some $L_2 >0$ such
  that $\mathcal{L}_{\lambda_1} + L_2 < L$,  we consider
${ \stackrel{\bullet}{\mathscr{U}}
 \subset  D(\zetadot,L_2) \setminus \Omega_{L} \subset D(\omegadot,L)
\setminus \Omega_{L} }$ such that $\stackrel{\bullet}{\mathscr{U}}$ is a
  star-shaped domain with respect to $\zetadot=\lambda_1(1)$. For $\xidot
\in \stackrel{\bullet}{\mathscr{U}}$, 
consider a path $\lambda_2$ starting from $\zetadot$, ending at $\xidot$ and whose
  image is the line segment $[\zetadot,\xidot]$. Then the product $\lambda_1
  \lambda_2$ belongs to $\mathfrak{R}_{\Omega_\star}(L)$ and we consider its
  $\Omega_\star$-homotopy class $\xi = \mathrm{cl}(\lambda_1  \lambda_2)$. We note
$\mathscr{U}$  the set of such points $\xi$.
\end{itemize}

We show that the system $\mathscr{B}= \{ \mathscr{U} \}$ made of these sets
 provides a basis for a topology on
$\mathscr{R}_{\Omega_\star}$.
 Obviously, every element $\xi  \in \mathscr{R}_{\Omega_\star}$
  belongs to at least one $\mathscr{U} \in \mathscr{B}$.
Now assume that $\xi \in \mathscr{U} \cap
  \mathscr{V}$, $\mathscr{U}, \mathscr{V} \in \mathscr{B}$.
\begin{itemize}
\item If $\xi = \omega$, then necessarily
$\stackrel{\bullet}{\mathscr{U}}$ and 
$\stackrel{\bullet}{\mathscr{V}}$ are two  star-shaped domains with
  respect to $\omegadot$, $\stackrel{\bullet}{\mathscr{U}}$ is a subset
  of $ D(\omegadot,L_1) \setminus \Omega_{L_1}^\star$ and  
$\stackrel{\bullet}{\mathscr{V}}$ is a subset
  of ${ D(\omegadot,L_2) \setminus \Omega_{L_2}^\star}$. Set $L = \max
  \{L_1, L_2\}$. Then $\stackrel{\bullet}{\mathscr{W}} =
\stackrel{\bullet}{\mathscr{U}} \cap 
\stackrel{\bullet}{\mathscr{V}} \subset
  D(\omegadot,L) \setminus \Omega_{L}^\star$ is
  also a  star-shaped domain with
  respect to $\omegadot$ to which is associated a $\mathscr{W} \in
  \mathscr{B}$ that satisfies:
$\xi \in \mathscr{W} \subset  \mathscr{U} \cap \mathscr{V}$.
\item Otherwise  $\xi \neq \omega$ and $\xi \in 
  \mathscr{U} \cap \mathscr{V}$. There is no loss of generality in
  assuming also that $\omega \notin \mathscr{U} \cap \mathscr{V}$. Thus :
\begin{itemize}
\item for some $L>0$, $\stackrel{\bullet}{\mathscr{U}}
 \subset  D(\zetadot_1,L_2)  \setminus \Omega_{L}$  is a
  star-shaped domain with respect to $\zetadot_1 = \lambda_1(1)$  where
  $\lambda_1 \in  \mathfrak{R}_{\Omega_\star}(L) $  satisfies $\displaystyle
  \mathcal{L}_{\lambda_1}+L_2 <L$. Also we have $\xi = \mathrm{cl}(\lambda_1
  \lambda_2)$ where
 $\lambda_2$ starts from $\zetadot_1$, ends at $\xidot$ and is such that
${ \underline{\lambda}_2([0,1]) = [\zetadot_1,\xidot] }$.
\item for some $L^\prime>0$, $\stackrel{\bullet}{\mathscr{V}}
 \subset  D({\zetadot}{_1^\prime},L_2^\prime) \setminus \Omega_{L^\prime} $  is a
  star-shaped domain with   respect to ${\zetadot}{_1^\prime} = \lambda_1^\prime(1)$
where $\lambda_1^\prime \in \mathfrak{R}_{\Omega_\star}(L^\prime) $ satisfies 
$\mathcal{L}_{\lambda_1^\prime} + L_2^\prime < L^\prime$.  Also $\xi =
\mathrm{cl}(\lambda_1^\prime \lambda_2^\prime)$ where
 $\lambda_2^\prime$ starts from ${\zetadot}{_1^\prime}$, ends at $\xidot$
and is such that  $\underline{\lambda}_2^\prime([0,1]) = [\zeta_1^\prime,\xi]$.
\end{itemize}
Now choose $L_2^{\prime \prime}>0$ such that $\displaystyle
 \mathcal{L}_{\lambda_1 \lambda_2} +L_2^{\prime \prime} <L$  and $\displaystyle
 \mathcal{L}_{\lambda_1^\prime \lambda_2^\prime} +L_2^{\prime \prime}
 <L^\prime$. Consider 
$\stackrel{\bullet}{\mathscr{W}} \subset \stackrel{\bullet}{\mathscr{U}}
 \cap \stackrel{\bullet}{\mathscr{V}} \cap D(\xidot,L_2^{\prime
   \prime})$  a star-shaped domain with 
 respect to $\xidot$. To $\stackrel{\bullet}{\mathscr{W}}$ is
 associated  $\mathscr{W} \in \mathfrak{B}$  and
$\xi \in \mathscr{W}  \subset \mathscr{U} \cap \mathscr{V}$.
\end{itemize}
We show that the topology thus defined by $\mathfrak{B}$ is
Hausdorff.  We consider two points
$\zeta $ and $\zeta^\prime$ in
$\mathscr{R}_{\Omega_\star}$. Clearly if $\zetadot \neq {\zetadot}{^\prime}$,
then $\zeta$ and $\zeta^\prime$ have disjoints neighbourhoods. Thus
assume that $\zetadot = {\zetadot}{^\prime}$ ($\neq \omegadot$,
say) but $\zeta \neq \zeta^\prime$. Suppose the existence of a neighbourhood  
$\mathscr{U} \in \mathfrak{B}$ of $\zeta$, a neighbourhood 
$\mathscr{U}^\prime \in \mathfrak{B}$  of
$\zeta^\prime$  such that  $\mathscr{U}  \cap \mathscr{U}^\prime \neq
\emptyset$. This means that
there exists $\xi \in \mathscr{U} \cap \mathscr{U}^\prime$ that satisfies:
\begin{itemize}
\item $\xi = \mathrm{cl}(\lambda_1 \lambda_2)$ with $\zeta= \mathrm{cl}(\lambda_1)$,
 $\lambda_2$ starting from $\zetadot$, ending at $\xidot$ with image 
  the line segment $[\zetadot,\xidot] \subset \stackrel{\bullet}{\mathscr{U}}$;
\item $\xi = \mathrm{cl}(\lambda_1^\prime \lambda_2^\prime)$ with 
$\zeta^\prime= \mathrm{cl}(\lambda_1^\prime)$,
 $\lambda_2^\prime $ starting from ${\zetadot}{^\prime} =\zetadot$, endings at $\xidot$
  the line segment $[\zetadot,\xidot] \subset
  {\stackrel{\bullet}{\mathscr{U}}}{^\prime}$ for its range. 
\end{itemize}
This implies that $\lambda_1$ and $\lambda_1^\prime$ are in the same
class, that is $\zeta = \zeta^\prime$ and we get a contradiction.

The topological space $\mathscr{R}_{\Omega_\star}$ is obviously 
(arc)connected. Also, by the very construction of the topology, for
every $\mathscr{U} \in \mathfrak{B}$, the restriction
$\mathfrak{p}|_{\mathscr{U}} : \zeta  \mapsto \zetadot \in
\stackrel{\bullet}{\mathscr{U}}$ is a 
homeomorphism. This means that $\mathscr{R}_{\Omega_\star}$ is an
\'etal\'e space on 
$\mathbb{C}$ (but of course not a covering space). 
We have thus shown the following proposition.

\begin{prop}\label{PFS-RS}
Let $\Omega_\star$ be is a discrete filtered set.
The (pointed) space  $\mathscr{R}_{\Omega_\star}$ is a topologically
(arc)connected separated space.  With the projection $\mathfrak{p}$, 
 the space  $\mathscr{R}_{\Omega_\star}$ is an \'etal\'e space on
$\mathbb{C}$.
\end{prop}

When pulling back by $\mathfrak{p}$ the complex structure of $\mathbb{C}$,
$\mathscr{R}_{\Omega_\star}$ becomes a Riemann surface. This Riemann
surface is the smallest in the following sense:

\begin{lemm}
Let $(\mathscr{R}^\prime,
\mathfrak{p}^\prime, \omega^\prime)$ be a Riemann surface over
$\mathbb{C}$, with $\mathfrak{p}^\prime(\omega^\prime)= \omegadot$. We
suppose that any $\Omega_\star$-allowed path can be lifted on
$\mathscr{R}^\prime$ from
$\omega^\prime$ with respect to $\mathfrak{p}^\prime$. Then
$(\mathscr{R}_{\Omega_\star},\mathfrak{p}, \omega)$ is contained in $(\mathscr{R}^\prime,
\mathfrak{p}^\prime, \omega^\prime)$, that is there exists a continuous map
$\tau: \mathscr{R}_{\Omega_\star} \to \mathscr{R}^\prime$ such that
$\mathfrak{p}^\prime \circ \tau = \mathfrak{p}$ and $\tau(\omega)=\omega^\prime$.
\end{lemm}

\begin{proof}
From the very definition of  the Riemann
surface  $(\mathscr{R}_{\Omega_\star}, \mathfrak{p})$ 
 associated with a  discrete filtered set $\Omega_\star$ centred
at $\omegadot \in \mathbb{C}$,  every $\Omega_\star$-allowed path can
be lifted with respect to  $\mathfrak{p}$ into a path on
$\mathscr{R}_{\Omega_\star}$ with initial point $\omega =
\mathfrak{p}^{-1}(\omegadot)$. 
Indeed, assume that $\underline{\lambda} \in
\mathfrak{R}_{\Omega_\star}(L)$ and define $\underline{\lambda}_t : s \in
[0,1] \mapsto \underline{\lambda}_t(s) = \underline{\lambda}(ts)$ 
for $t \in [0,1]$. Then $\underline{\lambda}_t \in
\mathfrak{R}_{\Omega_\star}(L)$ and the mapping
$\underline{\Lambda} : 
t \in [0,1] \mapsto \Lambda(t) = \mathrm{cl}(\underline{\lambda}_t )$
is continuous and is a lifting of $\underline{\lambda}$ from
$\omega$. 
This lifting is unique thanks to the uniqueness of
lifting \cite{For}. Pick any $\Omega_\star$-allowed path
$\lambda$. Its lifting with respect to $\mathfrak{p}$ from $\omega$
ends at $\zeta$ while its lifting with respect to
$\mathfrak{p}^\prime$ from $\omega^\prime$ ends at $\xi$. This gives a
mapping $\tau : \zeta \in \mathscr{R}_{\Omega_\star} \mapsto \xi \in
\mathscr{R}^\prime$ which is  well-defined and injective (uniqueness
of lifting), continuous (we work with \'etal\'e spaces) and preserves fibers.
\end{proof}

\begin{defi}\label{DFS-RS}
Let $\Omega_\star$ be a   discrete filtered set centred at
$\omegadot$. Any Riemann surface $(\mathscr{R}^\prime,
\mathfrak{p}^\prime, \omega^\prime)$ over
$\mathbb{C}$ which is isomorphic to $(\mathscr{R}_{\Omega_\star},
\mathfrak{p}, \omega)$, is called
 a Riemann surface associated with $\Omega_\star$.
\end{defi}

In this definition, isomorphic means the existence of a fiber preserving
homeomorphism $\tau: \mathscr{R}_{\Omega_\star} \to \mathscr{R}^\prime$.

We would like to point out a consequence of the topology considered on
these Riemann surfaces.
On Fig. \ref{Yafei-E-5}, we consider  a  discrete
    filtered set $\Omega_\star$  centred at $0$ such that, for $0 <
    L_1 < L_2$,  $\Omega_{L_1} = 
    \{0,\omega_1, \omega_2\}$, $\Omega_{L_2} = \Omega_{L_1} \cup
    \{\omega_3, \omega_4\}$. We have drawn two paths  $\lambda_1, \lambda_2 \in
    \mathfrak{R}_{\Omega_\star}(L_1)$ ending at the same point
    $\zeta$ and
    $\Omega_\star$-homotopic, $\mathrm{cl}(\lambda_1) = \mathrm{cl}(\lambda_2)$. Also we
    have drawn a path $\lambda_3$ starting at $\zeta$ and ending at
    $\xi$   such that the  product path $\lambda_2\lambda_3$ belongs to
    $\mathfrak{R}_{\Omega_\star}(L_2)$. We remark that
    $\lambda_1\lambda_3 \notin  \mathfrak{R}_{\Omega_\star}(L_2)$. \\
The path $\lambda_1$, {\em resp.}  $\lambda_2$, can be lifted with
respect to  $\mathfrak{p}$ to a path $\Lambda_1$, {\em resp.}  $\Lambda_2$, on
$\mathscr{R}_{\Omega_\star}$ with initial point $0=
\mathfrak{p}^{-1}(0)$ 
and common end point $\mathrm{cl}(\lambda_1) = \mathrm{cl}(\lambda_2)$. From
that point, $\lambda_3$ can be lifted with
respect to  $\mathfrak{p}$ to a path $\Lambda_3$ ending at $ \mathrm{cl}(\lambda_2
\lambda_3)$. We thus see that the path $\Lambda_1 \Lambda_3$ is well
defined and $\mathfrak{p}(\Lambda_1 \Lambda_3) = \lambda_1
\lambda_3$. This means that $\lambda_1 \lambda_3$ can be lifted on
$\mathscr{R}_{\Omega_\star}$ with
respect to  $\mathfrak{p}$ from  $0$ despite the fact that
$\lambda_1\lambda_3$ is not $\Omega_\star$-allowed. We say that
$\omega_3$ is a ``removable $\Omega_\star$-point'' for $\lambda_1\lambda_3$.

\begin{figure}[thp]
\centering\includegraphics[scale=.7]{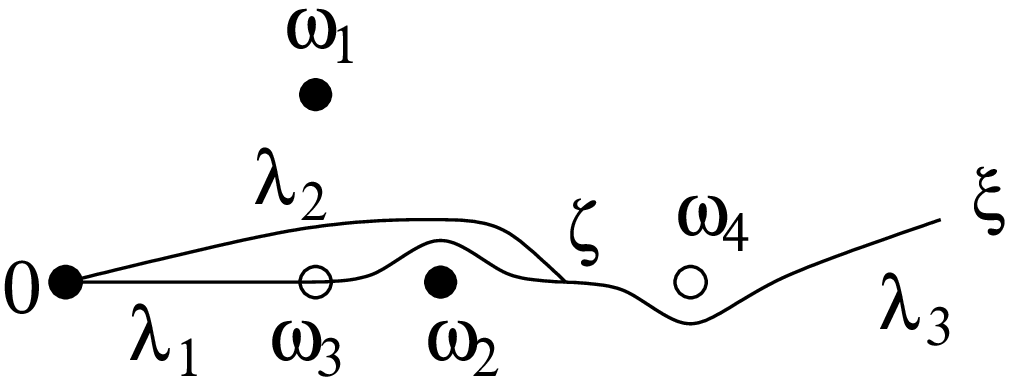}\\
  \centering\caption{
  }\label{Yafei-E-5}
\end{figure}

\begin{defi}\label{removablepoint}
Let $\Omega_\star$  be a discrete filtered set centred at
$\omegadot_0 \in \mathbb{C}$ and
$(\mathscr{R}_{\Omega_\star},\mathfrak{p})$ its associated Riemann
surface. Let $\Lambda$ be a piecewise $\mathcal{C}^1$ path on $\mathscr{R}_{\Omega_\star}$
starting from $\omega_0 = \mathfrak{p}^{-1}(\omegadot_0)$ and $\lambda =
\mathfrak{p} \circ \Lambda$. Let $L>\mathcal{L}_{\lambda}$. If $\lambda$
meets a point $\omega \in \Omega_L$, then $\omega$ is called a
\textbf{removable $\Omega_\star$-point} for $\lambda$.
\end{defi}

\subsection{Distance of a path to $\Omega_\star$}

Let  $\Omega_\star$ be a discrete filtered set.
We consider a point $\zeta \in
\mathscr{R}_{\Omega_\star}$. For $r>0$ small enough and 
since a disc is a star-shaped domain with
respect to its origin, there exists a connected neighbourhood  
$\mathscr{U} \subset \mathscr{R}_{\Omega_\star}$  of
$\zeta$ so that
$\mathfrak{p}|_{\mathscr{U}} : \mathscr{U} \to D(\zetadot,r)$
is a homeomorphism.

\begin{defi}
For any $\zeta \in \mathscr{R}_{\Omega_\star}$ and for $r>0$ small
enough, the \textbf{ball}  $D(\zeta,r)$ centred at
$\zeta$ with radius $r$ is the connected neighbourhood of
$\zeta$ so that 
$\mathfrak{p}|_{D(\zeta,r)} : D(\zeta,r) \to D(\zetadot,r) \subset \mathbb{C}$
is a homeomorphism. 
The distance $\rho_{\Omega_\star}(\zeta)$ of $\zeta$ to $\Omega_\star$
is  the supremum of the $r>0$
such that $\zeta$ has a neighbourhood of the form $D(\zeta,r)$. 
\end{defi}

In other words, $\rho_{\Omega_\star}(\zeta)$ is the distance of
$\zeta$ (for the norm $|.|$) to the boundary of $\mathscr{R}_{\Omega_\star}$. 
If $\Omega_\star$ is centred at $\omegadot = \mathfrak{p}(\omega)$,
then of course $\rho_{\Omega_\star}(\omegadot) =
\rho_{\Omega_\star}(\omega)$ and there is no risk of
misunderstanding. Notice that the mapping $\zeta \in
\mathscr{R}_{\Omega_\star} \mapsto \rho_{\Omega_\star}(\zeta)$ is
continuous since $\mathfrak{p}$ is continuous. This implies that 
$t \in [0,1] \mapsto
\rho_{\Omega_\star}\big(\underline{\Lambda}(t)\big)$ is a continuous
mapping when $\Lambda$ on
$\mathscr{R}_{\Omega_\star}$ starting from $\omega$, thus 
$\inf_{t \in [0,1]} \rho_{\Omega_\star}\big( \Lambda(t)\big) >0$ by compactness.

\begin{defi}\label{dpathomega}
Let  $\Omega_\star$ be a discrete filtered set centred
at ${\omegadot = \mathfrak{p}(\omega) \in \mathbb{C}}$.
For any path $\lambda$ issued from $\omegadot$ that can be lifted to 
$\mathscr{R}_{\Omega_\star}$ with respect to $\mathfrak{p}$ from
$\omega$ into the path $\Lambda$, one defines the distance $d(\lambda,
\Omega_\star)$ of $\lambda$ to  $\Omega_\star$ by
${d(\lambda, \Omega_\star) = \inf_{t \in [0,1]} \rho_{\Omega_\star}\big( \Lambda(t)\big) >0}$.
\end{defi}

\subsection{Some properties of the Riemann surface
  $\mathscr{R}_{\Omega_\star}$}

\begin{prop}\label{PropdefbasicbisB}
The Riemann surface $\mathscr{R}_{\Omega_\star}$ associated with the
discrete filtered set $\Omega_\star$  is simply connected.
\end{prop}

Thus $\mathscr{R}_{\Omega_\star}$ is conformally equivalent to the
open unit disc as a consequence of the uniformization theorem.

\begin{proof}
Let $\Omega_\star$ be a  discrete filtered set $\Omega_\star$ centred
at $\omegadot \in \mathbb{C}$ and  let $(\mathscr{R}_{\Omega_\star}, \mathfrak{p})$ 
be its associated pointed  Riemann surface, $\omega =
\mathfrak{p}^{-1}(\omegadot)$. 
Pick a non-constant
closed curve $\Lambda$ on $\mathscr{R}_{\Omega_\star}$. We want to show that 
$\Lambda$ is null-homotopic. Since $\mathscr{R}_{\Omega_\star}$ is
arcconnected, one can suppose that $\Lambda$ starts and ends at
$\omega$  and  there is no loss of generality in assuming that
$\Lambda$ does not meet $\omega$ apart from its extremities. Also, up to making
a slight deformation of $\Lambda$ in its homotopy class, one can
assume that $\lambda =  \mathfrak{p} \circ \Lambda$ is
$\mathcal{C}^1$, with length $\mathcal{L}_\lambda <L$ for some $L>0$
and that $\underline{\lambda}|_{]0,1[}$ avoids $\Omega_L$.
One can write $\Lambda$ under the
form\footnote{$\Lambda_2^{-1}$ is the inverse path,
  $\underline{\Lambda}_2^{-1}(s) = \underline{\Lambda}_2^{-1}(1-s)$.}
$\Lambda = \Lambda_1 \Lambda_2^{-1}$ where both $\Lambda_1$,
$\Lambda_2$ are paths starting from $\omega$ and ending at a point 
$\zeta \in \mathscr{R}_{\Omega_\star}$ with $\zeta \neq \omega$. 
We set $\lambda_1 =  \mathfrak{p} \circ \Lambda_1$ and 
$\lambda_2 =  \mathfrak{p} \circ \Lambda_2$. Both $\lambda_1$,
$\lambda_2$ are  $\Omega_\star$-allowed paths, precisely their belong
to $\mathfrak{R}_{\Omega_\star}(L)$. The path  $\lambda_1$, \underline{resp}. 
$\lambda_2$, can be lifted with respect to  $\mathfrak{p}$ from 
$\omega$ and, by uniqueness of lifting, corresponds to $\Lambda_1$, \underline{resp}. 
$\Lambda_2$. This implies that $\lambda_1$ and $\lambda_2$ are 
$\Omega_\star$-homotopic with $\zeta = \mathrm{cl}(\lambda_1) =
\mathrm{cl}(\lambda_2)$. The $\Omega_\star$-homotopy between
$\lambda_1$ and $\lambda_2$ can be lifted with respect to
$\mathfrak{p}$ and this provides a homotopy between $\Lambda_1$ and 
$\Lambda_2$. Therefore, $\Lambda = \Lambda_1 \Lambda_2^{-1}$ is null-homotopic.
\end{proof}

\begin{prop}\label{Propequiv}
Let $(\mathscr{R}_{\Omega_\star}, \mathfrak{p})$ be the Riemann
surface associated with a  discrete filtered set $\Omega_\star =
\Omega_\star (\omega)$ centred at  $\omegadot$.
Then, for every $\zeta  \in \mathscr{R}_{\Omega_\star}$, there
exists a  discrete filtered set $\Omega_\star (\zeta)$ centred at
$\zetadot = \mathfrak{p}(\zeta)$ such that
every $\Omega_\star (\zeta)$-allowed path starting from $\zetadot$ can be lifted
  on $\mathscr{R}_{\Omega_\star}$ from $\zeta$ with respect to $\mathfrak{p}$.
\end{prop}

\begin{proof}
We consider the Riemann
surface  $(\mathscr{R}_{\Omega_\star}, \mathfrak{p})$  associated with a
discrete filtered set $\Omega_\star$ centred 
at $\omegadot \in \mathbb{C}$ and set $\omega =
\mathfrak{p}^{-1}(\omegadot)$.  Pick a point
 $\zeta  \in \mathscr{R}_{\Omega_\star}$ with $\zeta \neq \omega$ and
 assume that $\zeta = \mathrm{cl}(\lambda_0)$, 
$\lambda_0 \in \mathfrak{R}_{\Omega_{L_0}}$ for some
$L_0>0$. We consider a path $\lambda$  ($\mathcal{C}^1$~piecewise) 
starting form $\zetadot$ and of length
$\mathcal{L}_\lambda$. 
If $\mathcal{L}_\lambda < \rho_{\Omega_\star}(\zeta)$, then  $\lambda$ can
be lifted from $\zeta$ with respect to $\mathfrak{p}$ : this is just a consequence
of the topology considered on  $\mathscr{R}_{\Omega_\star}$.\\
Assume now that $\lambda$ satisfies the properties :
$\underline{\lambda}(0) = \zetadot$, $\underline{\lambda}(]0,1])
\subset  \mathbb{C}\setminus \Omega_{L_0+L}$ and  $\mathcal{L}_\lambda
< L$ for some  $L>0$. For $\varepsilon >0$ small enough, one can construct
a $\Omega_\star$-homotopy
${ H : t \in [0,1] \mapsto H_t \in
  \mathfrak{R}_{\Omega_\star}  }$
such that
\begin{itemize}
\item $H_0 = \underline{\lambda}_0$;
\item for every $t \in [0,\varepsilon]$,  $H_t \in
   \mathfrak{R}_{\Omega_\star}(L_0)$ and $H_t(1) = \underline{\lambda}(t)$;
\item $H_\varepsilon$ belongs to  $\mathfrak{R}_{\Omega_\star}(L_0) \cap
  \mathfrak{R}_{\Omega_\star}(L_0+L)$;
\item for every $t \in [\varepsilon,1]$, $H_t \in
\mathfrak{R}_{\Omega_\star}(L_0+L)$ and  $H_t(1) = \underline{\lambda}(t)$.
\end{itemize}
Indeed, for $t \in [0,\varepsilon]$, $H_t$ realizes a small deformation
of $\underline{\lambda}_0$ so as to avoid the points of $\Omega_{L_0+L}$  while for $t
\in [\varepsilon,1]$, $H_t$ is for instance the standardized product of
$H_\varepsilon$ with $\underline{\lambda}|_{[\varepsilon, t]}$.
This $\Omega_\star$-homotopy can be lifted  with
respect to $\mathfrak{p}$ into a homotopy
$\mathcal{H}  : t \in [0,1] \mapsto \mathcal{H}_t$
where, for every $t \in [0,1]$, $\mathcal{H}_t : [0,1] \to
\mathscr{R}_{\Omega_\star}$ is a path starting at
$\omega$. Therefore, the path $\Lambda : t \in [0,1] \mapsto
\mathcal{H}_t(1) \in \mathscr{R}_{\Omega_\star}$ is a lifting of
$\underline{\lambda}$  from $\zeta$.
This has the following consequences. There exists a  discrete
filtered set $\Omega_\star(\zeta)$ centred at $\zetadot$, 
\begin{itemize}
\item for $L>0$ small enough, $\Omega_{L}(\zeta) =  \{\zeta\}$,
\item for $L>0$ large enough,
$\displaystyle  \Omega_{L}(\zeta) = \big( \Omega_{L_0+L} \cap D(\zetadot, L)
\big) \cup \{\zeta\} $
\end{itemize}
such that every $\Omega_\star(\zeta)$-allowed path can be lifted on
$\mathscr{R}_{\Omega_\star}$  with respect to $\mathfrak{p}$ from $\zeta$.
\end{proof}

\subsection{Seen and glimpsed points}\label{sectionseen}

We denote by $\mathbb{S}^1 \subset \mathbb{C}$  the circle of
directions\index{Directions of half-lines!on $\mathbb{C}$, $\mathbb{S}^1$} 
about $0$ of
half-lines  on $\mathbb{C}$. We  usually identify $\mathbb{S}^1$  with
$\mathbb{R}/2\pi\mathbb{Z}$.

\begin{defi}
Let  $I \subset \mathbb{S}^1$ be an open arc, $L>0$ and $\omega \in
\mathbb{C}$.  We denote by 
$\sect{_\omega^L}(I)$ the following open sector adherent to $\omega$:
$${\displaystyle \sect{_\omega^L}(I) = \{\zeta = \omega+
\xi e^{i\theta} \in 
\mathbb{C} \, \mid \, \theta \in I, \, 0 <\xi <L \}}.$$
\end{defi}

Assume that  $\Omega_\star$ is a discrete filtered set centred at
$\omega \in \mathbb{C}$,  $\theta \in \mathbb{S}^1$ is a given
direction and $L>0$.  Since $\Omega_L$ is a
finite set, observe that 
${\Omega_L \cap \sect{_\omega^L}(I) = \Omega_L \cap ]\omega, \omega+e^{i\theta}L[}$ when
$I=]-\alpha+\theta, \theta+\alpha[$ with $\alpha>0$ chosen small enough.

\begin{defi}
Let  $\Omega_\star$ be a discrete filtered set centred at
$\omega \in \mathbb{C}$,  $\theta \in \mathbb{S}^1$ and $L>0$. We
denote $\Omega_L^\star(\theta) = \Omega_L \cap ]\omega, \omega + e^{i\theta }L[$.
One says that $\alpha \in ]0,\frac{\pi}{2}[$ is 
a $\Omega_\star(\theta,L)$-angle if
$\Omega_L \cap \sect{_\omega^L}(I) = \Omega_L^\star(\theta)$
with $I=]-\alpha+\theta, \theta+\alpha[$.
\end{defi}

\begin{defi}\label{Allpaththeta}
Let  $\Omega_\star$ be a discrete filtered set centred at
$\omega \in \mathbb{C}$, $\theta \in \mathbb{S}^1$  and $L>0$.
We denote by $\mathfrak{R}(\Omega_\star, \theta,L)$ the set of
piecewise $\mathcal{C}^1$  paths $\lambda$ that satisfy the conditions:
\begin{itemize}
\item $\underline{\lambda}(0) = \omega$ and $\mathcal{L}_\lambda < L$;
\item  for every $t\in[0,1]$,  the right and left derivatives
  $\underline{\lambda}^\prime(t)$ do not  vanish;
\item there exists a $\Omega_\star(\theta,L)$-angle $\alpha \in
  ]0,\frac{\pi}{2}[$  such that for every $t \in [0,1]$,
${\mathrm{arg}  \, \lambda^\prime(t)  \in
  ]-\alpha+\theta, \theta +\alpha[}$.
\end{itemize}
\end{defi}

Remark that, apart from its origin, a path 
$\lambda \in \mathfrak{R}(\Omega_\star, \theta, L)$
stays in an open sector of the form
 $\sect{_\omega^L}(I)$, $I=]-\alpha+\theta, \theta+\alpha[$, with a
$\alpha$  a $\Omega_\star(\theta,L)$-angle. Moreover $\lambda$ always moves
forward in that sector.

\begin{prop}\label{propglimp}
Let  $\Omega_\star$ be a discrete filtered set centred at
$\omega_0 \in \mathbb{C}$ and $\theta \in \mathbb{S}^1$ a direction. 
There exists a uniquely defined discrete and closed set
$\mathrm{GLIMP}_{\Omega_\star}^\star(\theta) \subset \mathbb{C}$ that satisfies the
following conditions for any $L>0$:
\begin{itemize}
\item $\displaystyle
\mathrm{GLIMP}_{\Omega_\star}^\star(\theta,L)  \subset \Omega_L^\star(\theta)$, where
${\mathrm{GLIMP}_{\Omega_\star}^\star(\theta,L) =
  \mathrm{GLIMP}_{\Omega_\star}^\star(\theta) \cap D(0,L)}$;
\item  any path
belonging to $\mathfrak{R}(\Omega_\star, \theta, L)$ that
circumvents to the right or the left the set
$\mathrm{GLIMP}_{\Omega_\star}^\star(\theta,L)$, can be lifted on the
Riemann surface $(\mathscr{R}_{\Omega_\star},\mathfrak{p})$ 
with respect to $\mathfrak{p}$ from $\mathfrak{p}^{-1} (\omega_0)$.
\item when at least one point is removed from $\mathrm{GLIMP}_{\Omega_\star}^\star(\theta)$,
  then the above property is no more satisfied.
\end{itemize}
\end{prop}

\begin{proof}
We show proposition \ref{propglimp} by constructing
$\mathrm{GLIMP}_{\Omega_\star}^\star(\theta)$. \\
If $\bigcup_{L>0}  \Omega_L^\star(\theta)= \emptyset$, 
then $ \mathrm{GLIMP}_{\Omega_\star}^\star(\theta) = \emptyset$. Otherwise,
from the very definition of $\Omega_\star$, one can define an increasing
sequence $0 = L_{-1} < L_0 <  L_1 < L_2 < \cdots $ such that
\begin{itemize}
\item  $\Omega_{L_0}^\star(\theta) = \emptyset$, 
\item  for every $i \in \mathbb{N}^\star$, 
$\Omega_{L_{i}}^\star(\theta) = \Omega_{L_{i-1}}^\star(\theta) \cup
\{\omega_{i_1}, \cdots, \omega_{i_l}\}$ 
with $\{\omega_{i_1}, \cdots, \omega_{i_l}\}$ a finite subset of
$]\omega_0, \omega_0 +e^{i\theta}L_{i-1}]$; 
\item $\Omega_{L}^\star(\theta)=\Omega_{L_{i}}^\star(\theta)$ for every $L \in
]L_{i-1}, L_i]$ and every $i \in \mathbb{N}$.
\end{itemize}
We construct $\mathrm{GLIMP}_{\Omega_\star}^\star(\theta)$ by induction on
$i \in \mathbb{N}$.

\noindent \textbf{Case $i=0$}. Since $\Omega_{L_0}^\star(\theta) = \emptyset$, then 
for every $L \leq L_0$,  every  $\lambda \in \mathfrak{R}(\Omega_\star, L, \theta)$ is
  $\Omega_\star$-allowed and thus
 can be lifted on $\mathscr{R}_{\Omega_{\star}}$ with respect to
 $\mathfrak{p}$ from $\mathfrak{p}^{-1}(\omega)$. Therefore, we set
 $\mathrm{GLIMP}_{\Omega_\star}^\star(\theta,L) = \emptyset$ for $L \leq
 L_0$.

\noindent \textbf{Case $i=1$}. From the above property, if
${\omega \in \Omega_{L_1}^\star(\theta)}$ satisfied $|\omega| <
L_0$, then it  is a removable $\Omega_\star$-point for every path $\lambda \in
\mathfrak{R}(\Omega_\star, \theta, L)$ with $L \leq L_1$. Let us take $L \in
]L_{0}, L_1]$, so that  $\Omega_{L}^\star(\theta)  = \Omega_{L_{1}}^\star(\theta)$:
\begin{itemize}
\item either there is  $\omega  \in
  \Omega_{L_{1}}^\star(\theta)$ of the form $\omega = L_0
  e^{i\theta}$. 
In this case the condition \linebreak
${ \mathrm{GLIMP}_{\Omega_\star}^\star(\theta,L) = \mathrm{GLIMP}_{\Omega_\star}^\star(\theta,L_0) \cup
\{\omega\} }$ is both needed and sufficient  so as to ensure that for every $L
\leq L_1$, any path $\lambda \in \mathfrak{R}(\Omega_\star, 
\theta, L)$  that  circumvents $\mathrm{GLIMP}_{\Omega_\star}^\star(\theta,L)$
to the right or the left can be lifted on 
$\mathscr{R}_{\Omega_{\star}}$;
\item or  we set 
 ${\mathrm{GLIMP}_{\Omega_\star}^\star(\theta,L) =
 \mathrm{GLIMP}_{\Omega_\star}^\star(\theta,L_0) = \emptyset}$.
\end{itemize}

\noindent \textbf{Induction}. Pick some $i \in \mathbb{N}^\star$ and suppose that the
  following properties are valid  for every integer $j \in [1, i]$:
\begin{itemize}
\item for every $ L \in
  ]L_{j-1},L_j]$,  ${\mathrm{GLIMP}_{\Omega_\star}^\star(\theta,L) =
\mathrm{GLIMP}_{\Omega_\star}^\star(\theta,L_j) \subset
\Omega_{L_j}^\star(\theta)}$;
\item 
  $\mathrm{GLIMP}_{\Omega_\star}^\star(\theta,L_j) \cap ]\omega_0, \omega_0+
  e^{i\theta}L_{j-1}[ = \mathrm{GLIMP}_{\Omega_\star}^\star(\theta,L_{j-1})$;
\item  for every $L \leq L_i$, any path $\lambda  \in \mathfrak{R}(\Omega_\star, L,
\theta)$  that circumvents   $\mathrm{GLIMP}_{\Omega_\star}^\star(\theta,L)
$  to the right or the left, can be lifted on 
$\mathscr{R}_{\Omega_{\star}}$ with respect to $\mathfrak{p}$ from
$\mathfrak{p}^{-1}(\omega_0)$. 
\end{itemize}
From these properties,
every ${\omega \in \Omega_{L_{i+1}}^\star(\theta) \setminus
\mathrm{GLIMP}_{\Omega_\star}^\star(\theta,L_i)}$  such that $|\omega| <
L_i$ is a removable $\Omega_\star$-point for every path $\lambda \in
\mathfrak{R}(\Omega_\star, \theta, L)$ with $L \leq L_{i+1}$. 
From the fact that $\Omega_{L,\theta}^\star  =
\Omega_{L_{i+1},\theta}^\star$ for every $L \in ]L_{i}, L_{i+1}]$:
\begin{itemize}
\item either there is $\omega 
\in \Omega_{L_{i+1}}^\star(\theta)$ of the form $\omega= L_i
e^{i\theta}$. 
In that case one sets
$\mathrm{GLIMP}_{\Omega_\star}^\star(\theta,L) = \mathrm{GLIMP}_{\Omega_\star}^\star(\theta,L_i) \cup
\{\omega\}$ for $L \in ]L_i, L_{i+1}]$ and this provides a necessary
and sufficient condition to ensure that 
 any path  ${\lambda  \in \mathfrak{R}(\Omega_\star,\theta,L)}$ that
 circumvents 
$\mathrm{GLIMP}_{\Omega_\star}^\star(\theta,L)$  to the right or the left,
can be lifted on  $\mathscr{R}_{\Omega_{\star}}$,  for every $L \leq
L_{i+1}$;
\item or we simply set 
$\mathrm{GLIMP}_{\Omega_\star}^\star(\theta,L) =
\mathrm{GLIMP}_{\Omega_\star}^\star(\theta,L_i)$ for $L \in ]L_i, L_{i+1}]$.
\end{itemize}
This ends the proof.
\end{proof}

\begin{defi}\label{defseengl}
Let  $\Omega_\star$ be a discrete filtered set centred at
$\omega_0 \in \mathbb{C}$, $\theta \in \mathbb{S}^1$.  The discrete
and closed set\footnote{
The symbol $\prec$ stands for the total order on
$[\omega_0,
\omega_0+e^{i\theta}\infty[$ induced by $r \in [0,\infty] \mapsto
\omega_0 +re^{i\theta} \in [\omega_0,\omega_0+e^{i\theta}\infty[$. }
${\mathrm{GLIMP}_{\Omega_\star}^\star(\theta) = \{\omega_i \in ]\omega_0,
\omega_0+e^{i\theta}\infty[, \; \omega_0 \prec \omega_{1} \prec \omega_{2} \cdots \}}$
given by  proposition \ref{propglimp}
is the set of \textbf{glimpsed} $\Omega_\star$-points in the direction $\theta$. The
glimpsed point $\omega_1$  is the  \textbf{seen} $\Omega_\star$-point
in the direction~$\theta$. The \textbf{completed set of glimpsed}
$\Omega_\star$-points $\mathrm{GLIMP}_{\Omega_\star}(\theta)$ in the direction $\theta$ is defined by 
${\mathrm{GLIMP}_{\Omega_\star}(\theta) =
\mathrm{GLIMP}_{\Omega_\star}^\star(\theta) \cup \{\omega_0\}}$.
\end{defi}

\begin{rema}
The notion of glimpsed point can be defined in a
simpler way but the presentation we have made here is fitted
to the methods that we develop in the paper.
\end{rema}

\section{Endless continuability}\label{CNP-s-end}

\subsection{Endless Riemann surface}

\begin{defi}[Endless Riemann surface]\label{endRS}
A Riemann surface $(\mathscr{R}, \mathfrak{p})$, given as an \'etal\'e space on
$\mathbb{C}$, is said to be \textbf{endless} if for every $\zeta \in
\mathscr{R}$,  there
exists a discrete filtered set $\Omega_\star(\zeta)$ centred at
$\zetadot = \mathfrak{p}(\zeta)$  so that 
every $\Omega_\star(\zeta)$-allowed path can be lifted
  on $\mathscr{R}$ with respect to $\mathfrak{p}$ from $\zeta$.
\end{defi}

\begin{exem}
Let  $\Omega$ be a closed discrete subset of
$\mathbb{C}$. Then the universal covering $\widetilde{\mathbb{C}\setminus \Omega}$
of $\mathbb{C}\setminus \Omega$ is an endless Riemann surface.
\end{exem}

The following result is a direct 
consequence of proposition \ref{Propequiv}.

\begin{prop}
Let $\Omega_\star$ be a discrete filtered set. Then the associated
Riemann surface $(\mathscr{R}_{\Omega_\star}, \mathfrak{p})$ is endless.
\end{prop}

\subsection{Endless continuability}

\begin{defi}[Endless continuability]\label{defendlesscont}
A germ of holomorphic functions ${ \widehat{\varphi} \in \mathcal{O}_\omega}$ at $\omega
\in \mathbb{C}$ is 
endlessly continuable on $\mathbb{C}$ if $\widehat{\varphi}$ can be
analytically continued to an endless Riemann surface. One denotes by
$\hat{\mathscr{R}}_{\omega, \mathrm{endl}}$  the space of germ of
holomorphic functions at $\omega$ that are endlessly continuable on
$\mathbb{C}$. When $\omega=0$ we use the abridged notation 
$\hat{\mathscr{R}}_{\mathrm{endl}} = \hat{\mathscr{R}}_{0, \mathrm{endl}}$.
\end{defi}

\begin{prop}
A germ of holomorphic functions $\widehat{\varphi} \in \mathcal{O}_\omega$
at $\omega \in \mathbb{C}$ is endlessly continuable on $\mathbb{C}$ if
and only if  there exists a discrete filtered set
$\Omega_\star$ centred at $\omega$
such that $\widehat{\varphi}$ can be analytically continued along any
$\Omega_\star$-allowed path. 
\end{prop}

\begin{proof}
We suppose that $\widehat{\varphi} \in \mathcal{O}_{\omegadot}$ is
endlessly continuable, thus $\widehat{\varphi}$ can be analytically
continued to an endless Riemann surface $(\mathscr{R},
\mathfrak{p})$. This means that there exist  a neighbourhood
$\stackrel{\bullet}{\mathscr{U}} \subset \mathbb{C}$ 
of $\omegadot = \mathfrak{p}(\omega)$ and a neighbourhood
 $\mathscr{U} \subset \mathscr{R}$ of $\omega$ such that 
the restriction $\mathfrak{p}|_{\mathscr{U}} : \mathscr{U} \to
\stackrel{\bullet}{\mathscr{U}}$ is a homeomorphism, and there is a function 
$\Phi$ holomorphic on $\mathscr{R}$ so that $\phi = \Phi \circ
\mathfrak{p}|_{\mathscr{U}}^{-1}$ represents the germ
$\widehat{\varphi}$. By the very definition of an endless Riemann
surface, one can find a discrete  filtered set $\Omega_\star$ centred
at $\omegadot$ so that every $\Omega_\star$-allowed path 
$\lambda$ can be lifted with respect to $\mathfrak{p}$ into a path
$\Lambda$ starting at $\omega$. Since $\Phi$ can be analytically
continued along $\Lambda$, one gets that $\widehat{\varphi}$ can be
analytically continued along $\lambda$ as an upshot.

We now suppose that $\widehat{\varphi} \in \mathcal{O}_{\omegadot}$
can be analytically continued along any $\Omega_\star$-allowed path,
where $\Omega_\star$ is a discrete  filtered set $\Omega_\star$ centred
at $\omegadot$. By proposition \ref{Propequiv}, the Riemann surface
 $(\mathscr{R}_{\Omega_\star}, \mathfrak{p})$ associated with this
 discrete  filtered set is endless. To $\widehat{\varphi}$ is
 associated a germ of holomorphic functions $\Phi$ at $\omega$,
 $\omegadot = \mathfrak{p}(\omega)$ that can be analytically continued
 along the path $\Lambda$ starting from $\omega$ and deduced from any 
  $\Omega_\star$-allowed path $\lambda$. Since
 $\mathscr{R}_{\Omega_\star}$ is simply connected, this implies that
 $\Phi$ can be analytically continued to a function holomorphic on 
$\mathscr{R}_{\Omega_\star}$.
\end{proof}

\begin{defi}\label{defendlesscont-bis}
If $\Omega_\star$ is  a discrete filtered set centred at $\omega$, one
denotes by $\hat{\mathscr{R}}_{\Omega_\star}$ the space of germs of
holomorphic functions  at $\omega$ that can be analytically continued
to the Riemann surface ${\mathscr{R}}_{\Omega_\star}$.
\end{defi}

\subsection{Seen and glimpsed points}\label{sectionseenforvarphi}

We have introduced the notion of glimpsed $\Omega_\star$-points
(definition \ref{defseengl})
associated with a discrete filtered set  $\Omega_\star$ centred at $\omega_0$. If
$\widehat{\varphi}$ is an endless continuable germ at $\omega_0$ that
belongs to $\hat{\mathscr{R}}_{\Omega_\star}$ then, by the very
definition of $\mathrm{GLIMP}_{\Omega_\star}^\star(\theta)$,
$\widehat{\varphi}$ can be  analytically continued along any path that closely
follows the half-line $[\omega_0,
\omega_0+e^{i\theta}\infty[$ in the forward direction, while
circumventing to the right or to the left each point from the set
$\mathrm{GLIMP}_{\Omega_\star}^\star(\theta)$. However, this set is
not always the smaller one and one easily gets the following proposition.

\begin{prop}\label{propdefseenglgerms}
Let  ${ \widehat{\varphi} \in \hat{\mathscr{R}}_{\omega,
    \mathrm{endl}}}$ be an endlessly continuable
germ of holomorphic functions at $\omega \in \mathbb{C}$ and let
$\theta \in \mathbb{S}^1$ be a direction. There exists
a uniquely defined discrete and closed set
$\mathrm{GLIMP}_{\widehat{\varphi}}^\star(\theta) = \{\omega_i \in ]\omega_0,
\omega_0+e^{i\theta}\infty[, \; \omega_0 \prec \omega_{1} \prec \omega_{2} \cdots  \}$
such that:
\begin{itemize}
\item $\widehat{\varphi}$ can be  analytically continued along any path that closely
follows the half-line $[\omega_0,
\omega_0+e^{i\theta}\infty[$ in the forward direction, while circumventing
(to the right or to the left) each point of the set
$\mathrm{GLIMP}_{\widehat{\varphi}}^\star(\theta)$.
\item this property is no more valid if at least one point is removed
  from $\mathrm{GLIMP}_{\widehat{\varphi}}^\star(\theta)$. 
\end{itemize}
If $\Omega_\star$ a discrete filtered set centred at $\omega_0$, 
then ${\mathrm{GLIMP}_{\widehat{\varphi}}^\star(\theta) \subseteq
\mathrm{GLIMP}_{\Omega_\star}^\star(\theta)}$ for any
$\widehat{\varphi}$ belonging to $\hat{\mathscr{R}}_{\Omega_\star}$.
\end{prop}

\begin{defi}\label{defseenglgerms}
The elements of $\mathrm{GLIMP}_{\widehat{\varphi}}^\star(\theta)$  are called the \textbf{glimpsed}
singular points in the direction $\theta \in \mathbb{S}^1$ for the endlessly
continuable germ ${ \widehat{\varphi} \in \hat{\mathscr{R}}_{\omega,
    \mathrm{endl}}}$.  Specifically, $\omega_1$ is
the \textbf{seen} singular point in the direction $\theta$ for $\widehat{\varphi}$.
\end{defi}

\subsection{Continuability without cut}

We complete this Sect. with a brief comparison to  Ecalle's endless
continuability.  

\begin{defi}[Riemann surface without cut \cite{Ec85}]\label{defendcut}
Let $(\mathscr{R}, \mathfrak{p})$ be
a Riemann surface  given as an \'etal\'e space on
$\mathbb{C}$. This surface is said to be \textbf{without cut} if for every $\omega \in
\mathscr{R}$,  there exists a closed and
discrete set $\mathrm{sing} (\omega) \subset \mathbb{C}$ that
satisfies the following properties. Introducing $\omegadot = \mathfrak{p} (\omega)$: 
\begin{enumerate}
\item \label{pro1} 
if the line segment $[\omegadot, \omegadot^\prime] \subset \mathbb{C}$ does
not meet $\mathrm{sing}(\omega)$, then $[\omegadot, \omegadot^\prime]$ can be
lifted   homeomorphically on $\mathscr{R}$ with respect to $\mathfrak{p}$ from $\omega$.
\item \label{pro2} for every line segment $[\omegadot, \omegadot^\prime] \subset \mathbb{C}$ that
  meets the points $\omegadot_1, \cdots, \omegadot_r$ of
  $\mathrm{sing}(\omega)$,  there exists an open rectangle $W$ neighbourhood of
  $[\omegadot, \omegadot^\prime]$ such that each of the
  $2^r$ simply connected open sets $W_j$ deduced from $W$ by making lateral cuts at $\omegadot_1,
  \cdots, \omegadot_r$ (see Fig. \ref{Yafei-E-3}) can be lifted
  homeomorphically on $\mathscr{R}$ with respect to $\mathfrak{p}$ to an open set $ \mathcal{W}_j
  \subset \mathscr{R}$ containing $\omega$.
\item \label{condcut3} if at least one point is removed from $\mathrm{sing}(\omega)$,
  then properties \ref{pro1}-\ref{pro2} are no more satisfied.
\end{enumerate}
One says that the  point $\omegadot_1 \in
\mathrm{sing}(\omega)$  is \textbf{seen} from $\omega$ if
${[\omegadot, \omegadot_1] \cap \mathrm{sing}(\omega) = \{\omegadot_1\}}$.
Otherwise the  points $\omegadot_j \in \mathrm{sing}(\omega)$ 
are \textbf{glimpsed} from $\omega$.
\end{defi}

\begin{figure}[thp]
\centering\includegraphics[scale=.5]{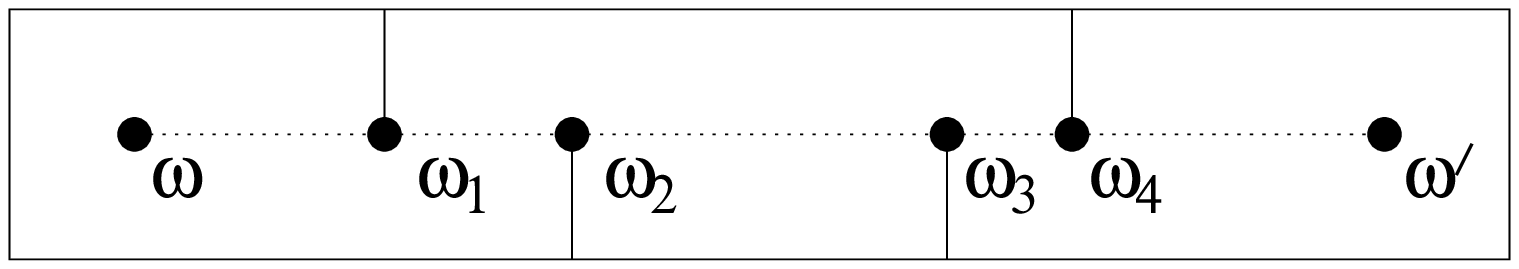}\\
  \centering\caption{}\label{Yafei-E-3}
\end{figure}

We note that in Definition \ref{defendcut}, condition \ref{condcut3}
is added so as to define the seen and glimpsed singular points.

\begin{defi}[Continuability without cut]
A germ of holomorphic functions $\widehat{\varphi} \in \mathcal{O}_0$ at $0
\in \mathbb{C}$ is said to be
\textbf{analytically continuable without cut} on $\mathbb{C}$ if its Riemann surface is
without cut.
\end{defi}

There is the following relationship between endless continuability in
the sense of definition \ref{endRS} and continuability without cut.

\begin{prop}
Let $(\mathscr{R}, \mathfrak{p})$ be a  Riemann surface. If $\mathscr{R}$
is endless, then $\mathscr{R}$ is without cut.
\end{prop}

\begin{proof}
We assume that the  Riemann surface $(\mathscr{R}, \mathfrak{p})$ is
endless. We consider a point  $\omega \in
\mathscr{R}$, $\omegadot = \mathfrak{p} (\omega)$ :   there
exists a discrete filtered set $\Omega_\star$ centred at $\omegadot$  such that
every $\Omega_\star$-allowed path $\lambda \in
\mathfrak{R}_{\Omega_\star}^\star$   can be lifted
  on $\mathscr{R}$  from $\omega$ with respect to $\mathfrak{p}$. We consider the line segment
  $[\omegadot, \omegadot^\prime]  \subset \mathbb{C}$ and  $L > l>0$  where $l= |\omegadot^\prime -
  \omegadot|$.
\begin{enumerate}
\item  Assume  that  the line segment $[\omegadot, \omegadot^\prime]$ does not
  meet $\Omega_L$ apart from $\omegadot$.  Then the path $\lambda : t \in
  [0,1] \mapsto \omegadot + t(\omegadot^\prime - \omegadot)$ belongs to
  $\mathfrak{R}_{\Omega_\star}(L)$  and thus can be lifted by
  $\mathfrak{p}$ from $\omega$.
\item  Assume  that the line segment $[\omegadot, \omegadot^\prime]$
  meets the points $\omegadot, \omegadot_1, \cdots, \omegadot_r$ of
  $\Omega_L$. Consider an open rectangle $W$ centred on (and thus neighbourhood of)
  $[\omegadot, \omegadot^\prime]$, of length  $l+2l^\prime$ and width
$2l^\prime$ where $l^\prime>0$ satisfies $l+2l^\prime >L$. For
$l^\prime$ small enough one has $W \cap \Omega_L = \{\omegadot, \omegadot_1, \cdots,
\omegadot_r\}$ and for every ${\zetadot  \in W \setminus \{\omegadot_1, \cdots,
\omegadot_r\}}$,  there exists a path $\lambda \in
\mathfrak{R}_{\Omega_\star}(L)^\star$ such that $\lambda([0,1]) \subset
W \setminus \{\omegadot_1, \cdots,
\omegadot_r\}$, ${\lambda(0) = \omegadot}$  and $\lambda(1) = \zetadot$.

Now assume that $W_j$ is one of the  $2^r$ simply connected open sets
$W_j$ deduced  from $W$ by making lateral cuts at $\omegadot_1,  \cdots,
\omegadot_r$ (see Fig. \ref{Yafei-E-3}). 
Then for every $\zetadot \in W_j$, there exists a path $\lambda \in
\mathfrak{R}_{\Omega_\star}(L)$ such that $\lambda([0,1]) \subset
W_j$, $\lambda(0) = \omegadot$  and $\lambda(1) = \zetadot$. This path can
be lifted with respect
to $\mathfrak{p}$ into a path starting from $\omega$ and ending at a point $\zeta$
such that $\mathfrak{p}(\zeta) = \zetadot$.  We note $\mathcal{W}_j$ the set of  these
points $\zeta$.\\
By its very definition, $\mathcal{W}_j$ is an open arcconnected subset of
$\mathscr{R}$ such that $\mathfrak{p}
(\mathcal{W}_j ) = W_j$. Moreover $\mathfrak{p}_{|\mathcal{W}_j}$ is
injective. Indeed, if one considers two paths $\lambda_0, \lambda_1 \in
\mathfrak{R}_{\Omega_\star}(L)$ such that $\lambda_1([0,1]) \subset
W_j$ and  $\lambda_2([0,1]) \subset
W_j$ and ending at the same point $\zetadot \in W_j$, 
one can easily constructs a $\Omega_\star$-homotopy ${\Gamma : t \in [0,1]
\mapsto \Gamma_t \in \mathfrak{R}_{\Omega_\star}(L)}$ between
$\underline{\lambda}_0$ and $\underline{\lambda}_1$, because $W_j$ is simply connected.
Finally,  since
$\mathfrak{p}$ is a local homeomorphism,   $\mathfrak{p}_{|\mathcal{W}_j}$ is a
homeomorphism between $\mathcal{W}_j$ and $W_j$.
\end{enumerate}
\end{proof}

\section{Endless continuability and convolution
  product}\label{mainsectionend}

For two germs $\widehat{\varphi}, \widehat{\psi} \in \mathcal{O}_\omega$ of holomorphic
functions at $\omega \in \mathbb{C}$, their convolution product
$\widehat{\varphi} \ast \widehat{\psi} \in \mathcal{O}_\omega$ is the germ of holomorphic
functions  at  $\omega$ defined  by the integral,
\begin{equation}\label{conprod1}
\widehat{\varphi} \ast \widehat{\psi} (\zeta) = \int_\omega^\zeta
\widehat{\varphi}(\eta)  \widehat{\psi}(\zeta+\omega-\eta) d\eta,
\end{equation}
for $\zeta$ close enough to $\omega$. 

\subsection{Endless continuability, stability under convolution
  product}

We state the main result of the paper.

\begin{theo}\label{thmendlessED}
Let $\widehat{\varphi}, \widehat{\psi} \in  \hat{\mathscr{R}}_{\omega,
    \mathrm{endl}}$ be two endlessly continuable germs of holomorphic
  functions at $\omega$. Then their convolution product $\widehat{\varphi}
  \ast \widehat{\psi}$ is endlessly continuable as well, $\widehat{\varphi}
  \ast \widehat{\psi} \in  \hat{\mathscr{R}}_{\omega, \mathrm{endl}}$.
More precisely, let $\Omega_\star$ and
$\Omega_\star^\prime$ be two discrete filtered sets centred at
$\omega$. If $\widehat{\varphi} \in
\widehat{\mathscr{R}}_{\Omega_\star}$ and $\widehat{\psi} \in
\widehat{\mathscr{R}}_{\Omega_\star^\prime}$, 
then their convolution product  $\widehat{\varphi}
  \ast \widehat{\psi}$ belongs 
 to $\widehat{\mathscr{R}}_{\Omega_\star \ast \Omega_\star^\prime}$
 where $\Omega_\star \ast \Omega_\star^\prime$ is the  fine
sum of the two discrete filtered sets.
\end{theo}

This theorem has an obvious but interesting corollary.

\begin{coro}
Let $\Omega_\star$ be a discrete filtered sets centred at
$\omega$ and $\widehat{\varphi} \in
\widehat{\mathscr{R}}_{\Omega_\star}$. Then the iterated convolution
products $\widehat{\varphi}^{\ast n}$ belong to
$\widehat{\mathscr{R}}_{\Omega_\star^\infty}$ with
$\Omega_\star^\infty$ the saturated of $\Omega_\star$.
\end{coro}

Theorem \ref{thmendlessED} is given in \cite{CNP1} and proved
there up to sometimes alluded-to key-points arguments.
The rest of this section is devoted to showing  this result rigorously. Our method 
differs from that of \cite{CNP1}. \\
Up to making a translation, one can suppose that $\omega=0$ and this is what we do in
the sequel. For ${\widehat{\varphi} \in
\widehat{\mathscr{R}}_{\Omega_\star}}$ and $\widehat{\psi} \in
\widehat{\mathscr{R}}_{\Omega_\star^\prime}$, notice that the
convolution product (\ref{conprod1}) provides a holomorphic function
on ${D(0, \rho_{\Omega_\star}(0)) \cap D(0,
  \rho_{\Omega_\star^\prime}(0))}$ since $\widehat{\varphi}$,
\underline{resp}. $\widehat{\psi}$ can be represented by a holomorphic
function on $D(0, \rho_{\Omega_\star}(0))$, \underline{resp}. $D(0,
  \rho_{\Omega_\star^\prime}(0))$.

\subsection{$(\Omega_\star \ast \Omega_\star^\prime)$-homotopy}

The following definition generalizes a definition from \cite{S012,
  S014}.

\begin{defi}
Let $\Omega_\star$, $\Omega_\star^\prime$ be two discrete filtered
sets centred at $0$,
$(\mathscr{R}_{\Omega_\star}, \mathfrak{p})$,
$(\mathscr{R}_{\Omega_\star^\prime}, \mathfrak{p}^\prime)$ their
associated Riemann surfaces. Let  
${H : (s,t) \in [0,1]^2 \mapsto H(s,t)=H_t(s) \in
\mathbb{C}}$ be a continuous map and
$H^\star : (s,t) \in [0,1]^2 \mapsto H^\star(s,t)=H_t^\star(s)  \in
\mathbb{C}$ the continuous map deduced from $H$ through the
identity\footnote{Remember that $H_t^{-1}(s)=H_t(1-s)$.}
 $H_t^\star(s) = H_t(1)-H_t^{-1}(s)$. One says that $H$ is 
 a \textbf{$(\Omega_\star \ast \Omega_\star^\prime)$-homotopy} if the following
 conditions are satisfied for every $t \in [0,1]$:
\begin{itemize}
\item $H_t(0) = 0$;
\item $H_t$ can be lifted with respect to $\mathfrak{p}$ on 
$\mathscr{R}_{\Omega_\star}$ from $0 = \mathfrak{p}^{-1}(0)$;
\item $H_t^\star$ can be lifted with respect to $\mathfrak{p}^\prime$ on 
$\mathscr{R}_{\Omega_\star^\prime}$ from $0 = {\mathfrak{p}^\prime}{^{-1}}(0)$;
\end{itemize}
The path $H_0$ is the initial path, $H_1$ is the final and the path $t
\in [0,1] \mapsto H_t(1)$ is the endpoint path of 
$H$.
\end{defi}

\subsection{Usefull lemmas}

We start with a technical lemma.

\begin{lemm}\label{techlem}
Let $H$ be a $(\Omega_\star \ast \Omega_\star^\prime)$-homotopy. Then 
$\inf_{t \in [0,1]} d(H_t, \Omega_\star) >0$ and $\inf_{t \in [0,1]}
d(H_t^\star, \Omega_\star^\prime) >0$. 
\end{lemm}

\begin{proof}
Let $(\mathscr{R}_{\Omega_\star}, \mathfrak{p})$ be the Riemann
surface associated with $\Omega_\star$. Since $H_t$ can be lifted 
with respect to $\mathfrak{p}$ on $\mathscr{R}_{\Omega_\star}$ from $0$ for every
$t \in [0,1]$ and using the homotopy lifting theorem, the ${(\Omega_\star \ast
  \Omega_\star^\prime)}$-homotopy   $H$ can be lifted with respect to
$\mathfrak{p}$ into a (unique) homotopy $\mathcal{H} :  (s,t) \in [0,1]^2
\mapsto  \mathcal{H}(s,t)=\mathcal{H}_t(s)$ such that
$\mathcal{H}_t(0) = 0$ for every $t \in [0,1]$.
 Since
the mapping $\zeta \in \mathscr{R}_{\Omega_\star} \mapsto d(\zeta,
\Omega_\star)$ is continuous, one concludes that $\inf_{(s,t) \in
  [0,1]^2}d(\mathcal{H}(s,t), \Omega_\star) >0$ by compactness. Thus
$\inf_{t \in [0,1]} d(H_t, \Omega_\star) >0$. The same reasoning holds
for $\inf_{t \in [0,1]} d(H_t^\star, \Omega_\star^\prime)$.
\end{proof}

\begin{lemm}\label{mainlem1forusuconv}
Let $\Omega_\star$, $\Omega_\star^\prime$ be two discrete filtered
sets centred at $0$ and  $\gamma$ be a piecewise $\mathcal{C}^1$
path such that $|\underline{\gamma}(0)| <
  \min \{\rho_{\Omega_\star}(0) , \rho_{\Omega_\star^\prime}(0)
  \}$. We suppose the existence of a 
${(\Omega_\star \ast \Omega_\star^\prime)}$-homotopy  $H$ whose endpoint
path is $\underline{\gamma}$ and such that ${H_0([0,1]) \subset
D(0, \rho_{\Omega_\star}(0))}$ and $H_0^\star([0,1]) \subset
D(0,\rho_{\Omega_\star^\prime}(0))$. Then, for any $\widehat{\varphi} \in
\widehat{\mathscr{R}}_{\Omega_\star}$ and any $\widehat{\psi} \in
\widehat{\mathscr{R}}_{\Omega_\star^\prime}$, their convolution product
 $\widehat{\varphi}
  \ast \widehat{\psi}$ can be analytically continued along $\gamma$.
\end{lemm}

\begin{proof}
Just adapt the proof of a similar lemma given in \cite{S012, S014} when $\Omega_\star$,
$\Omega_\star^\prime$ are closed discrete subsets of $\mathbb{C}$,
with the help of lemma \ref{techlem}.
\end{proof}

We now state the main lemma of this Sect.

\begin{lemm}[key-lemma]\label{mainlem2forusuconv}
Let $\Omega_\star$, $\Omega_\star^\prime$ be two discrete filtered
sets centred at $0$. Let $\lambda_0$ and $\gamma$ be two paths 
subject to the following conditions:
\begin{itemize}
\item  $\lambda_0$ satisfies $\underline{\lambda}_0 : s \in [0,1] \mapsto
  \underline{\lambda}_0 (s)=s\underline{\gamma} (0)$;
\item $\underline{\gamma}(0)$ satisfies $|\underline{\gamma}(0)| <
  \min \{\rho_{\Omega_\star}(0) , \rho_{\Omega_\star^\prime}(0) \}$;
\item the product path $\lambda_0 \gamma$ is 
$(\Omega_\star \ast \Omega_\star^\prime)$-allowed.
\end{itemize}
Then there exists a $(\Omega_\star \ast
\Omega_\star^\prime)$-homotopy  $H$ with endpoint path $\underline{\gamma}$ and
intial path  ${H_0 = \underline{\lambda}_0}$.
\end{lemm}

\begin{proof}
Part of our arguments comes from \cite{OU012, Del014}. We also use a
construction made in \cite{S012, S014} for the case where $\Omega_\star$,
$\Omega_\star^\prime$ are closed discrete subsets of $\mathbb{C}$,
and that simplifies the proof. The later is new up to our knowledge.

We first assume that the product path $\lambda_0 \gamma$ is 
$(\Omega_\star+\Omega_\star^\prime)$-allowed. Therefore, there exists
$L>0$ such that $\lambda_0 \gamma \in \mathfrak{R}_{\Omega_\star+
\Omega_\star^\prime}(L)$. In particular, $\mathcal{L}_{\lambda_0
\gamma} <L$ and $\gamma$ avoids the
set $(\Omega_\star + \Omega_\star^\prime)_L$,
$$(\Omega_\star + \Omega_\star^\prime)_L =  \{\zeta =
\omega+\omega^\prime \mid \omega \in \Omega_L, \, \omega^\prime \in
\Omega^\prime_L \mbox{ and } |\zeta| <L \}.$$
Making a slight deformation of $\gamma$ (an homotopy in
$\mathbb{C} \setminus (\Omega_\star + \Omega_\star^\prime)_L$ with
fixed extremities), we can assume that $\gamma$ is
$\mathcal{C}^1$. (There is no loss of generality with this assumption).

We pick two functions $\eta_{\Omega_L} : \mathbb{C} \to \mathbb{R}^+$ and 
$\eta_{\Omega_L^\prime} : \mathbb{C} \to  \mathbb{R}^+$, both
continuous and locally Lipschitz,  which furthermore satisfy: 
$$\{\zeta \in \mathbb{C} \mid
\eta_{\Omega_L}(\zeta)=0\} = \Omega_L, \hspace{2mm} \{\zeta \in \mathbb{C} \mid
\eta_{\Omega_L^\prime}(\zeta)=0\} = \Omega_L^\prime .$$
(For instance, $\eta_{\Omega_L}(\zeta) = d(\zeta,\Omega_L)$ where $d$
is the euclidean distance).  Remark that the mapping 
$\chi : (\zeta,t) \in \mathbb{C} \times [0,1] \mapsto
\eta_{\Omega_L}(\zeta)+\eta_{\Omega_L^\prime}(\underline{\gamma}(t)-
\zeta) \in  \mathbb{R}^+$ never vanishes : $\chi(\zeta,t)=0$ means 
$\zeta=\omega$ and $\gamma(t)-\zeta = \omega^\prime$ for some 
$\omega \in  \Omega_L$ and $\omega^\prime \in \Omega_L^\prime$, and
this implies $\gamma(t) = \omega + \omega^\prime$ which contradicts
the hypotheses made on 
$\lambda_0 \gamma$. This implies that the following non-autonomous vector
field,
$$X : (\zeta,t) \in \mathbb{C} \times [0,1] \mapsto X(\zeta,t) =
\frac{\eta_{\Omega_L}(\zeta)}{\eta_{\Omega_L}(\zeta)+\eta_{\Omega_L^\prime}(\underline{\gamma}(t)- 
  \zeta)}\underline{\gamma}^\prime (t)$$ 
is well-defined, continuous, everywhere locally Lipschitz with respect to
$\zeta$   and  bounded, $|X(\zeta,t)| \leq
|\underline{\gamma}^\prime (t)| \leq
\max_{[0,1]}|\underline{\gamma}^\prime|$. Therefore, its associated flow  ${ g_X :
  (t_0,t,\zeta) \in [0,1]^2\times 
\mathbb{C} \mapsto  g_X^{t_0,t}(\zeta) \in \mathbb{C} }$ is
$\mathcal{C}^1$ and globally defined as a consequence of the
Cauchy-Lipschitz theorem and the Gr\"onwall lemma.

We start with $H_0 = \underline{\lambda}_0$ and for every $t \in
[0,1]$, we consider the deformation $H_t$ of
$H_0$ along the flow $X$, precisely we set
$H_t = g_X^{0,t}(H_0)$. We get a mapping $H : (s,t)\in
[0,1]^2 \mapsto H(s,t) = H_t(s)$ with the following properties for
every $t \in [0,1]$ (check them or see \cite{S012, S014}):
\begin{itemize}
\item $H$ is of class $\mathcal{C}^1$;
\item $H_t(0)=0$ and $H_t(]0,1]) \in \mathbb{C} \setminus \Omega_L$;
\item $H_0 = \underline{\lambda}_0$ and 
the endpoint path $t \in [0,1] \mapsto H_t(1)$ coincides with
  the path $\underline{\gamma}$.
\end{itemize}
Let us now consider the family of paths $H^s : t \in [0,1] \mapsto
H^s(t)=H(s,t)$, for $s \in [0,1]$. These paths satisfy the following
properties. For every $s \in [0,1]$:
\begin{itemize}
\item $H^s$ is of $\mathcal{C}^1$-class, $H^s(0)= \underline{\lambda}_0(s)=H_0(s)$;
\item $\displaystyle \frac{dH^s(t)}{dt} = X\big(H^s(t),t \big)$, thus
  $\left|\frac{dH^s(t)}{dt}\right| \leq
  |\underline{\gamma}^\prime(t)|$ and this
  implies that $\mathcal{L}_{H^s} \leq \mathcal{L}_{\gamma}$;
\item  $H^0 \equiv 0$ and $H^s([0,1]) \subset \mathbb{C}  \setminus \Omega_L$ for $s \neq 0$;
\end{itemize}
The product of  paths $F^s = H_0|_{[0,s]} H^s$ is well-defined and has the following
properties, for any $s \in ]0,1]$: 
\begin{enumerate}
\item $\underline{F}^s$ is piecewise $\mathcal{C}^1$;
\item $\underline{F}^0 \equiv 0$ otherwise for $s>0$, 
 $\underline{F}^s(0)=0$ and $\underline{F}^s(]0,1]) \subset \mathbb{C}  \setminus \Omega_L$;
\item $\mathcal{L}_{F^s} =
  \mathcal{L}_{H_0|_{[0,s]}} + \mathcal{L}_{H^s}$, hence $\mathcal{L}_{F^s} 
  \leq \mathcal{L}_{\lambda_0 \gamma} <L$;
\end{enumerate}
Therefore for any $s \in ]0,1]$, $F^s$ belongs to
$\mathfrak{R}_{\Omega_\star}(L)$, thus is $\Omega_\star$-allowed and
can be lifted with respect 
to $\mathfrak{p}$ from $0$. This implies that $H_s$ can be lifted from
$\mathrm{cl}(\underline{\lambda}_0|_{[0,s]})$ with respect
to $\mathfrak{p}$ and this eventually provides a lifting $\mathcal{H}$ of
the mapping $H$.  One concludes that for every $t \in [0,1]$, the path
$H_t$ has a (unique) lifting $\mathcal{H}_t$ with respect to
$\mathfrak{p}$ from~$0$.

Look at the mapping $H^\star : (s,t) \in [0,1]^2 \mapsto
H^\star(s,t)=H_t^\star(s)$ deduced from $H$ by $H_t^\star(s) =
H_t(1)-H_t^{-1}(s)$. It is easy to see that the family of paths
$H_t^\star = g_{X^\star}^{0,t}(H_0^\star)$ is obtained by deformation  of $H_0^\star$, where
$g{_{X^\star} : (t_0,t,\zeta) \in [0,1]^2 \times \mathbb{C} \mapsto
  g_{X^\star}^{t_0,t}(\zeta)}$ is 
the flow associated with the non-autonomous vector
field,
$$X^\star : (\zeta,t) \in \mathbb{C} \times [0,1] \mapsto X(\zeta,t) =
\frac{\eta_{\Omega_L^\prime}(\zeta)}{\eta_{\Omega_L^\prime}(\zeta)+\eta_{\Omega_L}(\underline{\gamma}(t)-
  \zeta)}\underline{\gamma}^\prime (t).$$ 
The above reasoning can be applied as it stands for $H^\star$ and
$H_t^\star$ has a (unique) lifting $\mathcal{H}_t^\star$ with respect to 
$\mathfrak{p}^\prime$ from~$0$. 

We set ${H^\star}{^{s}}(t)=H^\star(s,t)$, thus ${H^\star}{^{(1-s)}}(t) = \gamma(t)-H^s(t)$.
From the identities $\displaystyle \frac{dH^s(t)}{dt} = X\big(H^s(t),t
\big)$ and
$\displaystyle \frac{d{H^\star}{^{(1-s)}}(t)}{dt} =
X_\star\big({H^\star}{^{(1-s)}}(t),t \big)$,
one easily gets~: $\displaystyle \left| \frac{dH^s(t)}{dt}\right|+
\left|
  \frac{{H^\star}{^{(1-s)}}(t)}{dt}\right|=|\underline{\gamma}^\prime(t)|$. The
 path ${H^\star}{^{(1-s)}}$ starts from the point ${{H^\star}{^{(1-s)}}(0)=
\underline{\lambda}_0(1-s)=H_0(1-s)}$, thus the product of paths 
${F^\star}{^{(1-s)}} = H_0|_{[0,1-s]} {H^\star}{^{(1-s)}}$ is
well-defined and
$$\mathcal{L}_{F^s}+\mathcal{L}_{{F^\star}{^{(1-s)}}}=\mathcal{L}_{\lambda_0
  \gamma} <L.$$
The upshot is that for any $s \in ]0,1[$, $F^s$ belongs to
$\mathfrak{R}_{\Omega_\star}(L_1)$ and ${F^\star}{^{(1-s)}}$ belongs to
$\mathfrak{R}_{\Omega_\star^\prime}(L_2)$ with $L_1+L_2 \leq L$. 
Therefore, only the points of the form $(\Omega_\star \ast
\Omega_\star^\prime)_L$ actually matter
for $\gamma$ to get the homotopies $H$ and $H^\star$. This
property allows to extend the above construction when 
the product path $\lambda_0 \gamma$ is 
$(\Omega_\star \ast \Omega_\star^\prime)$-allowed. Indeed, denote by
$K_L \subset \mathbb{C}\times [0,1]$
the subset made of the $(\zeta,t) \in \mathbb{C}\times [0,1]$ such
that $\zeta=\omega$, $\gamma(t)-\zeta = \omega^\prime$,
$\omega+\omega^\prime \in (\Omega_\star + \Omega_\star^\prime)_L
\setminus (\Omega_\star \ast
\Omega_\star^\prime)_L$. It is sufficient to remark that 
the restriction  $X|$ of $X$ to $\mathbb{C}\times
[0,1] \setminus K_L$ is still continuous, locally Lipschitz  and bounded, and the above
arguments show that the deformation $H_t = g_{X|}^{0,t}(H_0)$ of
$H_0$ along the flow $X|$ can be defined as well, and similarly for $H_t^\star$.
This ends the proof of the lemma.
\end{proof}

\subsection{The proof of theorem \ref{thmendlessED}}

Theorem \ref{thmendlessED} is a straightforward consequence  of lemma
\ref{mainlem1forusuconv} and lemma
\ref{mainlem2forusuconv}.

\subsection{Convolution product and  glimpsed  points}\label{submainthmcor}

The ideas developed in the proof of theorem \ref{thmendlessED}  can be easily adapted
to get the following informations on glimpsed points where, to
simplify, we only consider discrete filtered sets centred at~$0$.

\begin{prop}\label{mainthmglimpsed}
Let $\Omega_\star$, $\Omega_\star^\prime$ be two  discrete
filtered sets centred at $0$ and
$\mathrm{GLIMP}_{\Omega_\star}(\theta)$,
$\mathrm{GLIMP}_{\Omega_\star^\prime}(\theta)$ their
    respective completed sets of glimpsed points, for a given direction $\theta
    \in \mathbb{S}^1$. For any two endlessly continuable germs 
 $\widehat{\varphi} \in \widehat{\mathscr{R}}_{\Omega_\star}$ and $\widehat{\psi} \in
\widehat{\mathscr{R}}_{\Omega_\star^\prime}$, the set 
$\mathrm{GLIMP}_{\widehat{\varphi}\ast \widehat{\psi}}^\star(\theta)$
of  glimpsed singular points in the direction  $\theta \in \mathbb{S}^1$ for the
convolution product   $\widehat{\varphi}  \ast \widehat{\psi}$,
satisfies the condition : $\mathrm{GLIMP}_{\widehat{\varphi}\ast
  \widehat{\psi}}^\star(\theta) \subseteq
\{\mathrm{GLIMP}_{\Omega_\star}(\theta)+\mathrm{GLIMP}_{\Omega_\star^\prime}(\theta)\}
\setminus \{0\}$.
\end{prop}

\begin{proof}
It is sufficient to consider product paths $\lambda_0 \gamma$ of the
following form:
\begin{itemize}
\item  $\lambda_0$ satisfies $\underline{\lambda}_0 : s \in [0,1] \mapsto
  \underline{\lambda}_0 (s)=s\underline{\gamma} (0)$ and
  $\underline{\gamma}(0)$ satisfies the conditions :
  $\underline{\gamma}(0) \in  ]0, e^{i\theta}\infty[$ and
$|\underline{\gamma}(0)| < \min \{\rho_{\Omega_\star}(0) , \rho_{\Omega_\star^\prime}(0)\}$;
\item $\gamma$  avoids the set 
$\{\mathrm{GLIMP}_{\Omega_\star}(\theta)+\mathrm{GLIMP}_{\Omega_\star^\prime}(\theta)\}
\setminus \{0\}$;
\item $\lambda_0 \gamma$ belongs to $\mathfrak{R}(\Omega_\star,
  \theta, L) \cap
  \mathfrak{R}(\Omega_\star^\prime ,\theta,L)$  for some $L>0$;
\item $\gamma$ is of class $\mathcal{C}^1$, its derivative
  $\gamma^\prime$ do not vanish and there exists $\alpha
  \in ]0,\pi/2[$ small enough such that for every $t \in [0,1]$,
$\mathrm{arg} \, \gamma^\prime(t) \in
  ]-\alpha+\theta, \theta +\alpha[$.
\end{itemize}
We go back to the proof of the key-lemma \ref{mainlem2forusuconv}
where we replace   $\Omega_L$ by
$\mathrm{GLIMP}_{\Omega_\star}(\theta)$ and $\Omega^\prime_L$ by
$\mathrm{GLIMP}_{\Omega_\star^\prime}(\theta)$. We follow the
construction of  the
mapping $H$. It is easy to see that
for every $t \in [0,1]$, $\displaystyle \mathrm{arg} \,\frac{d
  H^s(t)}{dt} \in  ]-\alpha+\theta, \theta +\alpha[$. Defining $F^s$
like in the proof of lemma \ref{mainlem2forusuconv}, the upshot is
that  $F^s$ belongs to $\mathfrak{R}(\Omega_\star,
  \theta, L)$ and avoids $\mathrm{GLIMP}_{\Omega_\star}(\theta)$, for
  any $s \in ]0,1]$. This implies that the mapping $H$ has a (unique)
  lifting $\mathcal{H}$ with respect to $\mathfrak{p}$ with
  $\mathcal{H}_t(0)=0$ for every $t \in [0,1]$. The same result occurs for the mapping
  $H^\star$. One concludes with lemma \ref{mainlem1forusuconv}.
\end{proof}

One can draw the following consequences from 
both theorem \ref{thmendlessED} and
proposition~\ref{mainthmglimpsed}, where we use classical notations in
resurgence theory for which we refer to \cite{S014, Del014} :

\begin{coro}
The space of endlessly continuable functions
$\hat{\mathscr{R}}_{\mathrm{endl}}$ makes a differential convolution
algebra (without unit) on which the alien operators act. In
particular, if $\Omega_\star$, $\Omega_\star^\prime$ are two discrete
filtered sets centred at $0$, then for any  $\widehat{\varphi} \in
\widehat{\mathscr{R}}_{\Omega_\star}$, $\widehat{\psi} \in
\widehat{\mathscr{R}}_{\Omega_\star^\prime}$ and 
any $\omega \in  \renewcommand{\arraystretch}{0.5}
\begin{array}[t]{c}
\displaystyle \mathbb{C}^1\\
{\scriptstyle \bullet}
\end{array}
\renewcommand{\arraystretch}{1}$, 
 the alien operator
$\Delta_{\omega}^+$ acts on
$\stackrel{\triangledown}{\varphi}\ast
\stackrel{\triangledown}{\psi}$ with
$\stackrel{\triangledown}{\varphi} =
\!^\flat\widehat{\varphi}$, $\stackrel{\triangledown}{\psi} =
\!^\flat\widehat{\psi}$ and the following identity holds:
\begin{equation}\label{deralienplussing}
\Delta_{\omega}^+ (\stackrel{\triangledown}{\varphi} \ast \stackrel{\triangledown}{\psi} )=
(\Delta_{\omega}^+ \stackrel{\triangledown}{\varphi} )
\ast \stackrel{\triangledown}{\psi}
+\sum_{\omegadot_1+\omegadot_2=\omegadot}\big(\Delta_{\omega_1}^+
\stackrel{\triangledown}{\varphi}  
\big)\ast\big(\Delta_{\omega_2}^+ \stackrel{\triangledown}{\psi} \big)+
\stackrel{\triangledown}{\varphi} \ast \big(\Delta_{\omega}^+
\stackrel{\triangledown}{\psi} \big).
\end{equation}
In (\ref{deralienplussing}), the sum runs over all $\omegadot_1 \in
\mathrm{GLIMP}_{\Omega_\star}^\star(\thetadot)$,  
$\omegadot_2 \in
\mathrm{GLIMP}_{\Omega_\star^\prime}^\star(\thetadot)$ with $\thetadot
= \dot{\pi}(\theta)$ and  $\theta = \mathrm{arg}(\omega)
\in \renewcommand{\arraystretch}{0.5}
\begin{array}[t]{c}
\displaystyle \mathbb{S}^1\\
{\scriptstyle \bullet}
\end{array}
\renewcommand{\arraystretch}{1}$.
\end{coro}

\section{Conclusion}\label{Hend}

This article contributes to the resurgence theory in showing
rigorously the stability under convolution product 
of endlessly continuable functions, thus adds a
piece to the very foundation of this theory. We mention
that the notion of endless continuability used in this paper is less general that this
in \cite{CNP1} and \textit{a fortiori}  the endless continuability of
Ecalle.  We do not know whether our method could be applied to these
more general frames or not, however we know no application where such a generality
is needed. 

Since theorem  \ref{thmendlessED} brings in fine sums of discrete
filtered sets, series like $\sum_{n} a_n \widehat{\varphi}^{\ast n}$,
$a_n \in \mathbb{C}$ can be defined on the endless Riemann surface
$\mathscr{R}_{\Omega_\star^\infty}$  provided the uniform
convergence of the series
of any compact set of $\mathscr{R}_{\Omega_\star^\infty}$. Such a result is given in
\cite{CNP1} but for mistakes that have been corrected by Sauzin 
\cite{Sau013-3} for the case where $\Omega_\star$ stands for a closed
discrete subset of~$\mathbb{C}$. Considering the natural link between our
method and \cite{Sau013-3}, it is likely that Sauzin's
work can be generalized to endlessly continuable functions.

Finally, and like  mentioned in \cite{OU012, Sau013-3}, extensions of
theorem  \ref{thmendlessED} for the so-called weighted products
\cite{Ec93-1, Ec94} would be welcome so as to contribute to the knowledge on
the exact WKB analysis or coequational resurgence \cite{Ec84, DDP93,
  DDP97, DP99, Kawai-96, Kawai-004}, see also \cite{GdGvS014}. 
We hope to make some advances toward that direction in
a near future.

\backmatter

\end{document}